\documentclass[dvips,aos,preprint]{imsart}

\RequirePackage[OT1]{fontenc}
\RequirePackage{amsthm,amsmath}
\RequirePackage[numbers]{natbib}
\RequirePackage[colorlinks,citecolor=blue,urlcolor=blue]{hyperref}

\usepackage{amsmath,epsfig,epsf,psfrag, url}
\usepackage{amsmath, amsthm, amssymb}

\newcommand{\wt}{\widetilde}

\newcommand{\wh}{\widehat}

\newcommand{\ol}{\overline}
\newcommand{\ul}{\underline}

\newcommand{\bkone}{{\bf 1}}
\newcommand{\hatb}{\hat b}

\newtheorem{lemma}{Lemma}[section]
\newtheorem{theorem}{Theorem}[section]

\newtheorem{proposition}{Proposition}[section]

\startlocaldefs
\numberwithin{equation}{section}
\theoremstyle{plain}
\endlocaldefs

\begin{document}

\begin{frontmatter}
\title{Uniform Convergence and Rate Adaptive Estimation of a Convex Function}
\runtitle{Estimating convex functions}

\begin{aug}
\author{\fnms{Xiao} \snm{Wang}}
\and
\author{\fnms{Jinglai} \snm{Shen}}
\affiliation{Purdue University and University of Maryland Baltimore County}




\end{aug}

\begin{abstract}
This paper addresses the problem of estimating a convex regression function under both the sup-norm risk and the pointwise risk using B-splines. The presence of the convex constraint complicates various issues in asymptotic analysis, particularly uniform convergence analysis.
To overcome this difficulty, we  establish the uniform Lipschitz property of optimal spline coefficients in the $\ell_\infty$-norm by exploiting piecewise linear and polyhedral theory. Based upon this property, it is shown that this estimator attains optimal rates of convergence
on the entire interval of interest over the H\"older class under both the risks.  
In addition, adaptive estimates are constructed under both the sup-norm risk and the pointwise risk when the exponent of the H\"older class is between one and two. These estimates achieve a maximal risk within a constant factor of the minimax risk over the H\"older class.
\end{abstract}

\begin{keyword}[class=AMS]
\kwd[Primary ]{62G08}
\end{keyword}

\begin{keyword}
\kwd{Adaptive estimate}
\kwd{B-splines}
\kwd{convex regression}
\kwd{minimax risk}
\end{keyword}

\end{frontmatter}

\section{Introduction}

Consider the convex regression problem of the form
%
%
\begin{equation}
y_k = f(x_k) + \sigma  \epsilon_k, ~~~~ k=1, \ldots, n,
\end{equation}
where $f:[0, 1]\rightarrow \mathbb R$ is a convex function,  the $\epsilon_i$ are independent, standard normal errors, $x_i=i/n, i=1, \ldots, n$ are the design points.  Let $${\cal C}=\Big\{f: [0, 1]\rightarrow \mathbb R \, \Big | \,  {f(y)-f(x)\over y-x}\le {f(z)-f(y)\over z-y} \mbox{~if~}x<y<z\Big\}$$ be the collection of convex functions on $[0, 1]$. The goal of this paper is to estimate $f \in \mathcal C$ and analyze the performance of the estimate under both the sup-norm risk and the pointwise risk.

The shape restricted inference finds a wide range of  applications, and receives fast growing interest in diverse areas.
Examples include reliability (survival functions, hazard functions), medicine (dose-response curve), finance (option price and delivery price), and astronomy (mass functions).
Much effort has focused on monotone estimation via the least squares approach (i.e., Brunk's estimator) \cite{barlow_72,  brunk_58, mukerjee_88, robertson_88}.
For convex or concave regression, the least squares estimator was originally proposed in \cite{hildreth_54} and its asymptotic properties have been studied by \cite{dumbgen_04, groeneboom_01, hanson_76, mammen_91}. However, the least squares estimators suffer several major deficiencies: (i) they lack smoothness; (ii) they have a non-normal asymptotic distribution \cite{groeneboom_01, wright_81} with low convergence rates (e.g., of order $n^{1/3}$ for the Brunk's estimator) regardless of the smoothness of the true function; and (iii)
 they are inconsistent at boundary and
have a non-negligible asymptotic bias \cite{woodroofe_93}.

Other estimation procedures have also been developed for shape restricted inference. For instance,
Mammen and Thomas-Agnan \cite{mammen_99} studied constrained smoothing splines, but their computation is highly complicated; see \cite{ShenWang_CDC11} for a related result via control theoretic splines.
A two-step estimator was proposed in \cite{birke_06}: it isotonizes a derivative estimator and then obtains a convex one by integrating the monotone derivative. Meyer \cite{meyer_08} developed an algorithm
for cubic monotone estimation with an extension to convex constraints
and other variants, e.g., increasing-concave constraints.  A penalized monotone B-spline estimator was treated in \cite{ShenWang_SICON11}; its asymptotic behaviors were analyzed.
Additional results include \cite{hall_01, mukerjee_88,  ramsay_88, wang_10}, just to name a few.  In spite of the above mentioned progress, many critical questions remain open in convex regression and its asymptotic analysis, especially those related to adaptive estimation over a function class. One of bottle-neck difficulties in adaptive asymptotic analysis is largely due to the lack of uniform convergence properties of an estimator when a shape constraint is imposed.

In this paper, we consider estimation of a convex function in the H\"older class.
Let $H^r_L$ denote the H\"older class
$$H^r_L := \Big\{f: |f^{(\ell)}(x_1)-f^{(\ell)}(x_2)|\le L |x_1-x_2|^{\gamma}, ~~\forall x_1, x_2\in [0,1]\Big\},$$
where $\gamma = r-\ell\in (0, 1]$.  Let ${\cal C}_H(r, L) = {\cal C} \cap H^r_L$ be the collection of functions in both ${\cal C}$ and $H^r_L$. Since a convex function on $[0, 1]$ must be Lipschitz continuous, i.e.,  $\gamma=1$ and $\ell=0$, we have $r\ge 1$ for any $f\in {\cal C}(r, L)$.
It is well known that, for a fixed $r$, there exists an estimator, depending on $r$, which achieves the optimal rate of convergence in $H_{L}^r$ \cite{stone_82}. For example, the minimax sup-norm risk on $H_L^r$ has an asymptotic order given by
 \begin{equation} \label{eqn:converge_rate}
\inf_{\hat f}\sup_{f\in H_L^r} \mathbb E\big\{ \|\hat f - f\|_\infty\big\} \asymp L^{1\over 2r+1} \sigma^{2r\over 2r+1} \Big({\log n\over n}\Big)^{r\over 2r+1},
\end{equation}
where $a\asymp b$ means that $a/b$ is bounded by two positive constants from below and above. However, the existence of an adaptive estimator  (independent of $r$) that achieves the convergence rate in (\ref{eqn:converge_rate}) uniformly over $r$ is more subtle.
 When the sup-norm risk is considered, a series of papers, e.g., \cite{bertin_04,  donoho_94, korostelev_93,  lepski_00}, have shown that the kernel estimator can be used to construct such an adaptive estimator.
On the other hand, when the pointwise risk is considered, a full adaptive procedure achieving (\ref{eqn:converge_rate})  does not exist and a logarithmic penalty term must occur \cite{brown_96, lepski_90}.
Specifically, for any $x_0\in (0, 1)$, there exists a positive constant $\pi_1$ such that
\begin{equation}
\inf_{\hat f}\sup_{f\in H_L^r} \mathbb E\big\{ (\hat f(x_0) - f(x_0))^2\big\} \ge \pi_1 L^{2\over 2r+1} \sigma^{4r\over 2r+1} \Big({\log n\over n}\Big)^{{2r\over 2r+1}}.
\end{equation}
Other approaches for pointwise adaptive  estimation are reported in \cite{lepski_97, tsybabov_98}, where a similar phenomenon occurs. For general discussions of  adaptive methods for {\em unconstrained} functions, see \cite{Nemirovski_note, tsybakov_10} and the
references therein.

When a shape constraint is imposed, it was firstly noted in \cite{kiefer_82}  that it does not improve the optimal rate of convergence.  Further, it was found in \cite{low_02}  that the extra difference order constraint completely changes the adaptive estimation problem. In particular, Low and Kang \cite{low_02} proposed a pointwise rate adaptive procedure for monotone estimation in the minimax sense with respect to a Lipschitz parameter. Unfortunately, when this procedure is applied to an interval of fixed points, it does not yield a monotone function as an estimate. An adaptive monotone estimation procedure  is given in \cite{cator_11}, which studied the least squares estimator and showed that the attained rate of the probabilistic error is uniformly over a shrinking $L_2$-neighborhood of the true function. Other related papers on adaptive convex estimation include \cite{dumbgen_04}.

The present paper proposes a B-spline estimator with an arbitrary spline degree for convex regression. The convex shape constraint of an estimator is converted into the similar constraint on spline coefficients.  In addition to its conceptual simplicity and numerical efficiency, the obtained B-spline estimator is globally convex,  smooth by choosing a suitable spline degree, and attains boundary consistency (as well as at the interior) by selecting a proper number of spline bases. The major part of the paper is devoted to adaptive asymptotic analysis of the B-spline estimator on $\mathcal C_{H}(r, L)$ under both the sup-norm and pointwise risks. Toward this end,  it is essential and critical to establish certain uniform convergence properties of the B-spline estimator.  However,  challenging issues arise  due to the presence of constraints. For example, the closed form of optimal spline solutions  does not exist in general. Instead, they are characterized by complementarity conditions  \cite{FPang03_book, ShenPang_SICON05}  that give rise to a  nonsmooth piecewise linear function of observation data.
%
%
Due to the nonsmooth and combinatorial nature of complementarity problems, a thorough understanding of complementarity conditions and the associated piecewise linear function is far from trivial.  In this paper, we exploit optimization techniques, along with adaptive asymptotic statistical tools, to tackle these problems. The major contributions of the paper are:

  1. As a key technical contribution of the paper, we establish the uniform Lipschitz property of optimal spline coefficients with respect to the $\ell_\infty$-norm via piecewise linear and polyhedral theory  (cf. Theorem~\ref{theorem:uniform_Lip}).  Unlike the conventional and generic Lipschitz property in the $\ell_2$-norm (which is trivial to show), the attained Lipschitz property  in the $\ell_\infty$-norm requires a nontrivial argument that takes full advantage of the convex shape constraint. It yields a uniform sup-norm bound on variations of spline coefficients regardless of the number of spline bases,  leading to more precise and less conservative error estimates in uniform convergence analysis. This property paves the way for  asymptotic analysis   (e.g., cf. Propositions~\ref{lem:bias}--\ref{lem:pointwise}) and construction of adaptive procedures.

  2.   By exploring the uniform Lipschitz property, we obtain the following results in adaptive asymptotic analysis:
  \begin{enumerate}
\item[(2.1)] For a fixed order $r$,  the proposed B-spline estimator achieves an optimal minimax rate of convergence on ${\cal C}_{H}(r, L)$ under both the sup-norm and pointwise risks (cf. Section~\ref{sect:optimal_rate}). This result gives rise to an optimal choice of the number of spline bases.  Unlike the widely studied least squared convex estimator, the B-spline estimator achieves optimal convergence rates on the entire interval $[0, 1]$ under both the sup-norm and pointwise risks (cf. Theorem~\ref{thm:convergence}), thus leading to uniform consistency on $[0, 1]$.
        \item[(2.2)]  Adaptive estimators are constructed under both the sup-norm and pointwise risks over ${\cal C}_{H}(r, L)$ with $r\in [1, 2]$. These estimates achieve a maximum risk within a constant factor of the minimax risk over the H\"older class (cf. Section~\ref{sect:adaptive_est}). In particular, the pointwise adaptive estimator attains convexity on the interval $[0, 1]$ as well as the minimax risk over an entire range of values of $r\in [1, 2]$ and $L$.
        \item[(2.3)] A brief discussion on variance estimation is given in Section~\ref{sect:effective_dim}.
 %
\end{enumerate}

The paper is organized as follows. Section~\ref{sect:Opm_conditions} formulates the B-spline  convex estimator and develops optimality conditions for spline coefficients. The main results of the paper are presented in Section~\ref{sect:main_results}, including the uniform Lipschitz property and its implications in adaptive asymptotic analysis. Potential extensions and future research directions are discussed in Section~\ref{sect:discussion}. 
The technical proofs of the main results are given in Sections~\ref{sect:proof_optimality_PLform}--\ref{sect:proof_sig}.


%

\section{Formulation and Optimality Conditions}
\label{sect:Opm_conditions}

Denote the
$p\,$th degree B-spline basis with knots $0 = \kappa_0 < \kappa_1
< \cdots < \kappa_{K_n} = 1$ by $\big\{ B^{[p]}_k : k = 1, \ldots, K_n + p \big\}$. For simplicity, we consider equally
spaced knots, namely, $\kappa_1=1/K_n,\kappa_2=2/K_n, \ldots,
\kappa_{K_n}=1$. The value of $K_n$ will depend upon $n$ as
discussed below. Assume that $n/K_n$ is an integer denoted by
$M_n$. We consider the following convex spline estimator:
\[
  \hat{f}^{[p]}(x)=\sum_{k=1}^{K_n+p} \hat{b}_k B^{[p]}_k(x),
\]
where the spline coefficients $\hat{b}=\{\hat{b}_k, k=1, \ldots,
K_n+p\}$  minimize
\begin{equation}\label{equ:p}
\sum_{i=1}^n\Big( y_i - \sum_{k=1}^{K_n+p}b_k
B_{k}^{[p]}(x_i)\Big)^2
\end{equation}
subject to the convex constraint $\Delta^2 b\ge 0 $, where
$\Delta$ is the backward difference operator such that $\Delta b_k = b_k-b_{k-1}$ and $\Delta^2 = \Delta\Delta$.

Let the $n\times (K_n+p)$ design matrix $X = \big[B^{[p]}_k(x_j)
\big]_{j, k}$ and  denote $\beta_n = \sum_{i=1}^n \big(
B_k^{[p]}(x_i) \big)^2$ for $k=p+1, \ldots, K_n$. Given a spline degree $p$,
$\big( \beta_n \frac{K_n}{n}  \big)$ converges to a positive constant (depending on $p$ only)
as $(n/K_n) \rightarrow \infty$. Thus
there exists a positive constant $C_{\beta,p}$ (depending on $p$ only) such that
\begin{equation} \label{eqn:beta_n}
      \beta_n \ge  C_{\beta, p} \cdot \frac{n}{K_n}, \qquad \forall \ n, K_n.
\end{equation}

Define the
positive definite matrix $\Lambda_p:= X^T X/\beta_n \in \mathbb
R^{(K_n+p)\times (K_n+p)}$ and $\bar y := X^T y/\beta_n$, where
$y=(y_1, \ldots, y_n)^T$ (we drop the subscript $p$ in $\Lambda_p$
for notational simplicity). It is easy to verify that for a given
spline degree $p$, $\Lambda$ is a $(2p+1)$-banded matrix. For
instance, when $p=1$, $\Lambda$ is tridiagonal.
%
%
The convex constraint on spline coefficients is characterized by the following polyhedral
cone
\[
\Omega := \big\{b \in \mathbb R^{K_n+p}: b_k-2b_{k+1}+b_{k+2}\ge 0, \
k=1, \ldots, K_n+p-2\big\}.
\]
%
%
%
When the knots are equally spaced, it is easy to verify that if
the B-spline coefficient vector $\hat b$ is in $\Omega$, then
$\hat f^{[p]}$ is a convex function. Formulating (\ref{equ:p}) in
matrix notation, the underlying optimization problem becomes the
following equivalent constrained quadratic program
\begin{equation} \label{equ:bm2}
  \hat b =
   \arg\min_{b\in \Omega} \, {1\over 2}\, b^T \Lambda \, b -
   b^T\bar{y}.
\end{equation}

%
We first give the characterization of optimality conditions for
$\hat b$. The conditions are represented by complementarity
conditions, which plays a crucial role in addressing analytic and
statistical properties of the estimator. We provide a short
introduction of the complementarity condition. Two vectors
$u=(u_1, \ldots, u_d)^T$ and $v=(v_1, \ldots, v_d)^T$ in $\mathbb
R^d$ are said to satisfy the {\em complementarity condition}
\cite{CPStone_book92} if $u_i \ge 0$, $v_i \ge 0$, and $u_i
\,v_i=0$ for all $i=1, \cdots, d$. This condition can be put in a
more compact vector form: $0 \, \le \, u \ \perp \ v \, \ge \, 0$,
where $u \perp v$ means that the two vectors are orthogonal, i.e.,
$u^T v = 0$.

%
%

We introduce additional notation. Let
\[
     C = \left[\begin{array}{ccccccccc}
   1 & 0 & 0  &0& \cdots & 0 & 0\\
    1 & 1 & 0 &0 & \cdots & 0 & 0\\
   1 & 1 & 1 & 0 & \cdots & 0 & 0 \\
  & \cdots & & & \cdots & &\\
   1 & 1 &  1 &1 &\cdots & 1 &0\\
   1 & 1 &  1 &1 &\cdots & 1 &1
    \end{array} \right]\in \mathbb R^{ (K_n+p) \times (K_n+p)},
\]
and let $D_2 \in \mathbb{R}^{(K_n+p-2)\times (K_n+p)}$ be the 2nd-order difference matrix such that
$ D_2 b = [\Delta^2 (b_{3}), \cdots,
\Delta^2 (b_{K_n+p})]^T$; see (\ref{eqn:D_2}) for the explicit form of $D_2$.

\begin{theorem}\label{thm:char}
Let $C_{d \bullet}$ denote the $d$th row of $C$.
The necessary and sufficient conditions for $\hat b\in \Omega$ to minimize (\ref{equ:bm2}) are
\begin{equation} \label{equ:opt-m-1}
   0 \le D_2 \hat b \, \perp \,  C_{\gamma\bullet}\, C \big( \Lambda \hat b - \bar y\big) \ge 0,
\end{equation}
\begin{equation} \label{equ:opt-m-2}
C_{(K_n+p) \bullet}\, \big( \Lambda \hat b - \bar y\big) =
C_{(K_n+p) \bullet}\, C \big( \Lambda \hat b - \bar y\big)=0,
\end{equation}
where the index set $\gamma := \{1, \ldots, K_n+p-2\}$.
\end{theorem}

%
\subsection{Piecewise Linear Formulation of Optimal Spline Coefficients}
\label{subsect:PL_form}

It follows from Theorem~\ref{thm:char} that $\hat b(\bar y)$ is
characterized by the mixed complementarity conditions. It is known
from complementarity and polyhedral theory that $\hat b(\bar y)$
is a continuous piecewise linear function of $\bar y$ determined
by an index set $\alpha = \{ \, i \, | \, (D_2 \hat b)_i =0 \}
\subseteq \{1, \ldots, K_n+p-2 \}$ ($\alpha$ may be empty).
Indeed, $\hat b$ has $2^{(K_n+p-2)}$ linear selection
functions, each of which is denoted by $\hat b^\alpha$ corresponding to the index set $\alpha$. Hence,  the
solution mapping $\bar y \mapsto \hat b$ is a (continuous)
piecewise linear function with $2^{(K_n+p-2)}$ selection
functions.  The following proposition characterizes each linear selection function; its construction and proof is given in Section~\ref{sect:proof_PL_form}.

\begin{proposition}  \label{lem:PL_formulation}
 For each index set $\alpha \subseteq\{ 1, \ldots, K_n+p -2\}$, let $\ell:=K_n+p- |\alpha|$.
 Then  there exists a row independent matrix $F_\alpha \in \mathbb R^{\ell \times (K_n+p)}$ such that the linear selection function $\hat b^\alpha$ is given by
 \[
     \hat b^\alpha (\bar y)  \, = \, F^T_\alpha \big( \, F_\alpha \Lambda F^T_\alpha \, \big)^{-1} F_\alpha \bar y.
 \]
 %
%
\end{proposition}

In view of the above proposition and its construction (cf.  Section~\ref{sect:proof_PL_form}), a linear selection function
corresponds to an index set $\alpha$ depending on $\bar y$ (or $y$ by somewhat abusing notation).
Consequently, the piecewise linear function $\hat b$ can be written as
\[
\hat b (\bar y) = 
F^T_{\alpha(y)} \big(F_{\alpha(y)} \Lambda F^T_{\alpha(y)} \big)^{-1} F_{\alpha(y)} \bar y.
\]
Let $N(x): = \big[B_1^{[p]}(x), \ldots, B_{K_n+p}^{[p]}(x)\big]^T$. For a given $y$,
%
%
the convex B-spline estimator becomes
\begin{equation}\label{equ:ker}
\hat f^{[p]}(x) = N^T(x) \hat b(\bar y) = {1\over n}\sum_{i=1}^n N^T(x) F^T_{\alpha(y)} \Big(F_{\alpha(y)} {X^TX\over n} F^T_{\alpha(y)} \Big)^{-1} F_{\alpha(y)} N(x_i) y_i.
\end{equation}
Denote the weight function in (\ref{equ:ker}) by $K_\alpha(s, t)$, i.e.,
$$K_{\alpha(y)} (s, t) = N^T(s)F^T_{\alpha(y)} \Big(F_{\alpha(y)} {X^T(t) X(t) \over n} F^T_{\alpha(y)} \Big)^{-1} F_{\alpha(y)} N(t).$$
Hence, the convex spline estimator is a kernel estimator. However, the kernel depends on the index set $\alpha$, which in turn relies on the observation $y$. Therefore, the estimator is {\em not} a linear but a piecewise linear function in $y$.


%

\section{Main Results} \label{sect:main_results}

In this section, we exploit the piecewise linear formulation of
$\hat b$ to establish the uniform Lipschitz property of $\hat b$
in the $\ell_\infty$-norm. Roughly
speaking, this property says that $\hat b(\bar y)$ is a Lipschitz
function of $\bar y$ with a uniform Lipschitz constant (with
respect to the $\ell_\infty$-norm), regardless of
$K_n$ and $\alpha$. This property is critical in establishing
uniform consistency  and developing adaptive estimators. Formally, this property is stated in the following theorem whose proof is given in Section~\ref{sect:proof_uniform_Lip}.

\begin{theorem} \label{theorem:uniform_Lip}
  Given a spline degree $p$. There exists a positive constant $c_{\infty, p}$ (dependent on $p$ only) such that
  \begin{itemize}
   \item [(1)]  for any $K_n$ and any index set $\alpha$,
   $
            \big \| F^T_\alpha (F_\alpha \Lambda F^T_\alpha)^{-1}   F_\alpha \big \|_\infty  \, \le \, c_{\infty, p};
    $
  \item [(2)]
  for any $K_n$, 
       $
          \big  \| \hat b( u) - \hat b(v) \big \|_\infty  \, \le \, c_{\infty, p} \| u - v\|_{\infty}, \ \forall \  u, v \in \mathbb R^{K_n+p}.
       $
  \end{itemize}
\end{theorem}
  %
%

In the next, we apply the uniform Lipschitz property  to derive optimal rates of convergence in Section~\ref{sect:optimal_rate}, construct adaptive estimators under both sup-norm risk and pointwise risk in Section~\ref{sect:adaptive_est}, and study variance estimation in Section~\ref{sect:effective_dim}.

%
\subsection{Optimal Rate of Convergence} \label{sect:optimal_rate}

For any $f\in {\cal C}_H(r, L)$ with $r>1$, we write $\hat f_{(r)} := \hat f^{[p]}$ when using the spline degree $p= \lceil r-1\rceil$ to fit the data. If $r=1$, then $\hat f_{(r)}: = \hat f^{[p]}$ with $p=1$, namely, $\hat f_{(r)}$ is a piecewise linear spline. In the following, for a function $g:[0, 1] \rightarrow \mathbb R$, let $\| g\|_\infty:=\sup_{t\in[0, 1]} | g(t) |$.

\begin{theorem} \label{thm:convergence}
Assume $f\in {\cal C}_H(r, L)$. Then,
\begin{itemize}
\item[(1)] If $K_n$ is chosen as
$$K_n = \Big({L\over \sigma}\Big)^{2\over 2r+1} \Big({n\over \log n}\Big)^{1\over 2r+1},$$
then there exists a positive constant $\tilde C_{1r}$ dependent only on $r$ such that
\begin{equation}
\sup_{f\in {\cal C}_H(r, L)} \mathbb E\Big( \|\hat f_{(r)} - f\|_\infty  \Big) \le \tilde C_{1r} L^{1\over 2r+1} \sigma^{2r\over 2r+1}\Big({\log n\over  n}\Big)^{r\over 2r+1}.
\end{equation}
\item[(2)] For any $x_0\in [0, 1]$, if $K_n$ is chosen as
$$K_n = \Big({L\over \sigma}\Big)^{2\over 2r+1} n^{1\over 2r+1},$$
then there exists a positive constant $\tilde C_{2r}$ dependent only on $r$ such that
\begin{equation}
\sup_{f\in {\cal C}_H(r, L)} \mathbb E\Big( |\hat f_{(r)}(x_0) - f(x_0)|^2  \Big) \le \tilde C_{2r} L^{2\over 2r+1} \sigma^{4r\over 2r+1}n^{2r\over 2r+1}.
\end{equation}
\end{itemize}
\end{theorem}

It is known that the maximum likelihood estimate of a convex function is inconsistent at the boundary, which is called the {\it spiking problem} \cite{woodroofe_93}. In contrast, Theorem \ref{thm:convergence} shows that $\hat f_{(r)}$ is uniformly consistent on $[0, 1]$. The optimal choice of $K_n$ is of order $(n/\log n)^{1 \over (2r+1)}$ and $\|\hat f_{(r)} - f\|_\infty$ achieves the optimal rate of convergence, which is of order $(n/\log n)^{r \over (2r+1)}$ \cite{Nemirovski_note}. Under the pointwise risk, the optimal choice of $K_n$ is of order $n^{1 \over (2r+1)}$ and the estimator thus achieves the optimal rate of convergence, which is of order $n^{2r \over (2r+1)}$ \cite{stone_82}.

The next result shows that, for any $f\in {\cal C}_H(r, L)$ with $r>2$, the constrained spline estimator and the unconstrained spline estimator coincide with probability tending to one, provided that $f''(x)>0$ for all $x\in [0, 1]$.

\begin{theorem}\label{thm:uncon}
Assume $f\in {\cal C}_H(r, L)$ with $r>2$ and $f''(x)\ge c>0$ for all $x\in [0,1]$. Let $\hat f^{uc}$ be the unconstrained regression spline estimator. If $n^{-1}K_n^5 \log n\rightarrow 0$ and $K_n\rightarrow\infty$ as $n\rightarrow\infty$, then,
\[
P\big( \, \hat f^{uc}(x) = \hat f(x), ~\forall~ x\in [0, 1] \big) \longrightarrow 1.
\]
\end{theorem}

\begin{proof}
Zhou et al. \cite{zhou_00} studied the problem of estimating derivatives of a regression function using the corresponding derivatives of regression splines without shape constraint. For any $x\in [\kappa_k, \kappa_{k+1}]$, $k=0, \ldots, K_n-1$, if $\ell \ge 3$,
$$\mathbb E\Big( {d^2\over dx^2} \hat f^{uc}(x) \Big) -  f''(x) = b(x) + o(K_n^{-r+2}),$$
where
$$b(x) = {K_n^{-(\ell-2)} \over (\ell-1)!} \Big(f^{(\ell)}(\kappa_{k+1})-f^{(\ell)}(\kappa_{k}) \Big) B_{\ell-1}\big( K_n( x - \kappa_k)\big)$$
is of order $O(K_n^{-r+2})$, and $B_m(\cdot)$ is the $m$th Bernoulli polynomial inductively defined as follows:
$$B_0(x)=1, ~~~~~B_i(x) = \int_0^x i B_{i-1}(z)dz+b_i,$$
where $b_i = -i\int_0^1\int_0^xB_{i-1}(z)dzdx$ is the $i$th Bernoulli number.
The variance of ${d^2\over dx^2} \hat f^{uc}(x)$ is of order $n^{-1}K_n^{5}$ (cf.  \cite[Lemma 5.4]{zhou_00}). Similar to Lemma~\ref{lem:random-2} given in Section \ref{sect:proof_thm:sup-norm}, it can be shown that
 $$\Big \|{d^2\over dx^2} \hat f^{uc}(x) - {d^2\over dx^2} \mathbb E(\hat f^{uc}(x)) \Big \|_\infty = O_p\Big( \sqrt{n^{-1}K_n^5 \log n} \Big).$$
 Therefore, if $n^{-1}K_n^5 \log n\rightarrow 0$ and $K_n\rightarrow\infty$ as $n\rightarrow\infty$, $\|{d^2\over dx^2} \hat f^{uc} - f''\|_\infty=o_p(1)$.  Hence, the unconstrained and constrained estimators are asymptotically equivalent.
\end{proof}

%


\subsection{Adaptive Estimation} \label{sect:adaptive_est}

In this section, we construct adaptive estimators, with respect to both the sup-norm risk and thepointwise risk . These estimates have maximum risks within a constant factor of the minimax risk over ${\cal C}_H(r, L)$. We will focus on the function class ${\cal C}_H(r, L)$ with $1\le r\le 2$, where the differences between the constrained and unconstrained estimate do no vanish asymptotically.

%
\subsubsection{Adaptive Estimation under Sup-norm Risks}

It follows from Propositions \ref{lem:bias} and \ref{lem:random} that, for any $r\in [1, 2]$, the bounds for the bias and the stochastic term, respectively, are
\begin{eqnarray}
\sup_{f\in {\cal C}_H(r, L)} \|\bar f_{(r)} -  f\|_\infty &\le & C_1 L K_n^{-r},   \label{equ:bias2}\\
P\Big( \|\hat f_{(r)} - \bar f_{(r)}\|_\infty \ge u \Big) &\le & (K_n+1)\exp\Big\{ -{n\over 2 K_n C_{2}^2\sigma^2}u^2 \Big\},\label{equ:random2}
\end{eqnarray}
where $C_1$ and $C_2$ are two positive constants independent of $r \in[1,2]$. Hence  the optimal number of knots is
\begin{equation} \label{eqn:K_r}
K_{(r)} = \Bigg({C_2\over C_1 r\sqrt{2(2r+1)}}\Bigg)^{-{2\over 2r+1}} \Big({\sigma\over L} \Big)^{-2 \over 2r+1} \Big({\log n\over n}\Big)^{-1 \over 2r+1},
\end{equation}
and the optimal rate of convergence is
\begin{eqnarray}
\psi_{(r)} & = &  C_1 L K_{(r)}^{-r} + \sqrt{2\over 2r+1}~C_2\sigma\sqrt{\frac{ K_{(r)} \log n}{n} } \label{eqn:psi} \\
&=&  2 C_1^{1\over 2r+1} \Bigg({C_2\over r\sqrt{2(2r+1)}}\Bigg)^{{2r\over 2r+1}} L^{1\over 2r+1}\sigma^{2r\over 2r+1 } \Big({\log n\over n}\Big)^{r \over 2r+1 }. \nonumber
\end{eqnarray}

Given $n$, let $\tau_n :=\lceil  (\log n)^{1/2} \rceil$,
%
%
and $r_j := 1+j/\tau_n$, $j=0, 1, \ldots, \tau_n$ be the elements in $[1, 2]$. We consider the adaptive estimator using the idea of Lepski \cite{lepski_92}. Let
$$\hat k = \sup\Big\{0\le k\le \tau_n: \|\hat f_{(r_k)} - \hat f_{(r_j)}\|_\infty \le {1+\sqrt{2} \over 2}~\psi_{(r_j)} , \mbox{~ for any}~~j\le k\Big\}.$$
Define $\hat r := r_{\hat k}$. We use $\hat f_{(\hat r)}$ for estimation in the sup-norm distance.

\begin{theorem}\label{thm:sup-norm}
The estimator $\hat f_{(\hat r)}$ is a rate adaptive estimator on ${\cal C}_H(r, L)$ for the sup-norm distance, i.e., there exists a positive constant  $\pi_2$ such that
$$\sup_{r\in [1,2]}\sup_{f\in {\cal C}_H(r, L)} \mathbb E \big\{\|\hat f_{(\hat r)} - f\|_\infty\big\} \le \pi_2~L^{1\over 2r+1 }\sigma^{2r\over 2r+1}\Big({\log n\over n}\Big)^{r\over 2r+1}.$$
\end{theorem}

%
\subsubsection{Pointwise Adaptive Estimation}

In this section, we construct an estimator which attains the minimax rate of convergence for a whole range of values of $r\in [1,2]$ and $L$. In the context of convex regression, unlike the earlier work on pointwise adaptive estimation, a fully adaptive procedure can be obtained.

We explore the idea of Low and Kang \cite{low_02} to construct an adaptive estimate of $f(x_0)$ for any given  $x_0\in (0, 1)$.
Given the observation data $(y_i)^n_{i=1}$, let
\begin{equation} \label{eqn:grouped_y_k}
\bar y_k : = {\sum_{i=1}^n y_i I(\kappa_{k-1}<x_i\le \kappa_k)\over \sum_{i=1}^n I(\kappa_{k-1}<x_i\le \kappa_k)}, ~~k=1,\ldots, K_n.
\end{equation}
Then $\hat b=(\hat b_1, \ldots, \hat b_{K_n})^T$ minimizes $\sum_{k=1}^{K_n} (b_k - \bar y_k)^2 $
%
%
subject to the convex constraint $\Delta^2 b_k\ge 0$, $k=3,\ldots, K_n$. This indicates that  a piecewise constant spline with $p=0$ is used to fit the data in (\ref{equ:p}).
Recall that $M_n=n/K_n$.
Let
$$\zeta_k := {1\over n} \Big[ (k-1)M_n + {M_n+1\over 2}\Big].$$  Hence, $\zeta_k$ is the average of the design points on $(\kappa_{k-1}, \kappa_k]$.
Let $\wt f$ denote the piecewise linear function which interpolates $(\zeta_k, \hat b_k)$, $k=1, \ldots, K_n$.


Fix $x_0\in (0, 1)$. For each $n$, let $d_n \in \mathbb N$ satisfy $\zeta_{d_n}<x_0\le \zeta_{d_n+1}$.
Let $K_{n, j} := 2^j n^{1/5}$, where $K_{n, j}$ depends on $j$. Further, we let $\bar y_{k, j}$ denote the $\bar y_k$ defined in (\ref{eqn:grouped_y_k}) corresponding to a given $K_{n, j}$,
and let $\wt f_j$ be the estimator $\wt f$ corresponding to  $K_{n, j}$.
%
%
Fix a real number $\lambda >0$ such that $P(Z>\lambda)<1/4$, where $Z$ is a standard normal random variable.
Set
\[
I_j := I\Big(\Delta\bar y_{d_n+4, j} - \Delta \bar y_{d_n-2, j}\le \lambda 2^{ {j \over 2}+1} n^{-{2 \over 5} }\sigma \Big)\prod_{i=0}^{j-1} I\Big(\Delta\bar y_{d_n+4, i} - \Delta \bar y_{d_n-2, i} > \lambda 2^{{i \over 2}+1} n^{-{2 \over 5} } \sigma\Big).
\]
Note that exactly one $I_j\neq 0$ and thus the collection $\{I_j\}$ provides a selection procedure for $K_{n, j}$.
The adaptive estimator is given by
\begin{equation}\label{equ:ada}
  \wt f(x_0) = \sum_{j=1}^\infty \wt f_j(x_0) I_j.
\end{equation}

\begin{theorem} \label{thm:pointwise}
 The estimator in (\ref{equ:ada}) is a rate adaptive estimator under the pointwise risk, i.e., for any $x_0\in (0, 1)$, there exists a positive constant  $\pi_3$  such that
\begin{equation}
\sup_{r\in [1,2]}\sup_{f\in {\cal C}_H(r, L)} \mathbb E|\wt f(x_0) - f(x_0)|^2 \le \pi_3 L^{2\over (2r+1)} \sigma^{4r\over (2r+1)} n^{-2r\over (2r+1)}.
\end{equation}
\end{theorem}

%

\subsection{Variance Estimation} \label{sect:effective_dim}

In practice, the variance $\sigma^2$ is replaced by the estimated variance $\hat\sigma^2$ in the above adaptive procedures. We will briefly study the  asymptotic properties of the maximum likelihood estimator of $\sigma^2$. 
Given the observation data $y=(y_1, \ldots, y_n)^T \in \mathbb R^n$ at design points $x=(x_1, \ldots, x_n)^T \in \mathbb R^n$, let
$\hat f_y := \big( \hat f^{[p]} (x_1), \ldots, \hat f^{[p]}(x_n)\big)^T$ with $p=\lceil r-1 \rceil$ and
$\vec f :=  \big( f(x_1), \ldots, f(x_n)\big)^T$. Let
$\alpha(y)$ be an index set
corresponding to the optimal coefficient $\hat b(X^T y/\beta_n)$ defined in
Section~\ref{subsect:PL_form}. Then for fixed $K_n$ and $p$, we have $\hat f_y
= A_{\alpha(y)} y$, where
\begin{equation} \label{eqn:A_alpha}
 A_{\alpha(y)} = X F_{\alpha(y)}^T\Big( F_{\alpha(y)} X^T X F_{\alpha(y)}^T
\Big)^{-1}F_{\alpha(y)} X^T \in \mathbb R^{n\times n}.
\end{equation}
It follows from the similar discussion as in
Section~\ref{subsect:PL_form} that $\hat f_y:\mathbb R^n \rightarrow
\mathbb R^n$ is a continuous piecewise linear function, where each
linear selection function is defined by $A_{\alpha(y)}$. 
The MLE of $\sigma^2$ is
$\hat\sigma^2 = \big \|y - A_{\alpha(y)} y \big \|_2^2/n.$

\begin{theorem}\label{thm:sig} Assume $f\in {\cal C}(r, L)$,  $K_n\rightarrow \infty$ as $n\rightarrow \infty$, and let $p=\lceil r-1 \rceil$. If $K_n = o(n)$, then $\hat\sigma^2 \rightarrow \sigma^2$ in probability, and if $K_n = o(\sqrt{n})$, then $\sqrt{n}\big(\hat\sigma^2 - \sigma^2)$ is asymptotically normal with mean zero and variance $2\sigma^4$. Furthermore, if $K_n$ is of order $n^{1\over 2r+1}$, then $\displaystyle  |\mathbb E ( \hat\sigma^2 - \sigma^2 )|  = O\big(n^{-2r\over 2r+1}\big)$.
%
%
\end{theorem}


Variance estimation based on the differences of successive points has been studied originally by Rice \cite{rice_84}. Compared to this estimation and other generalizations, the MLE $\hat\sigma^2$ has a smaller asymptotic variance but a slightly larger bias. Meyer and Woodroofe \cite{meyer_00} studied the bias reduction variance estimator for a monotone regression model using the least squares method. The bias reduction variance estimator for a convex spline model is nontrivial and shall be addressed in a future paper.

%

%
\section{Discussion} \label{sect:discussion}

We have considered the B-spline estimators for convex regression in this paper. The proposed estimator and asymptotic analysis techniques can be extended to other shape restricted inference problems. For example, it is known that the uniform Lipschitz property (in the $\ell_\infty$-norm) holds for the monotone constraint \cite{ShenWang_SICON11, wang_10}. Therefore the minimax optimal convergence rates and adaptive estimators can be established in a similar manner. It is conjectured that  the uniform Lipschitz property holds for a higher order difference constraint. However, its development is much more involved and shall be reported in the future.

We have provided  optimal rate adaptive estimators for convex regression under both the sup-norm and the pointwise risks. Nonetheless, the question of  explicit construction of asymptotically exact adaptive estimation over the H\"older classes remains open. Other interesting research directions include adaptive confidence bands and hypothesis tests in convex estimation \cite{dumbgen_03, dumbgen_01}.


%


%
%
\section{Proofs for Section~\ref{sect:Opm_conditions}} \label{sect:proof_optimality_PLform}

%
\subsection{Proof of Theorem \ref{thm:char}} \label{sect:proof_optimality}

\begin{proof}[Proof of Theorem \ref{thm:char}]
Write (\ref{equ:bm2}) as $\min_{b\in
  \Omega} g(b)$, where the objective function $g(b):= \frac{1}{2}b^T \Lambda b - b^T \bar y$. It is clear that $g$ is
   coercive on $\mathbb R^{K_n+p}$ and strictly convex on the closed convex set $\Omega$.
   This ensures the existence and uniqueness of an optimal
   solution. Furthermore, since $\Omega$ is a  polyhedral
   cone, it is finitely generated by $\{ v^1, - v^1, v^2, -v^2, v^{3}, v^4, \ldots, v^{K_n+p} \}$.
   Here, for each $k=3, \ldots, K_n+p$,
   %
\[
   v^k = \Big( \, \underbrace{0, \, \ldots, \,
   0,}_{(k-1)-\mbox{copies}} v^k_k, \ldots, v^k_{K_n+p} \, \Big)^T
    = \Big( \, \underbrace{0, \, \ldots, \,
   0,}_{(k-1)-\mbox{copies}} 1, 2, \ldots, K_n+p-k+1 \, \Big)^T,
\]
and for $k=1, 2$,
\begin{equation} \label{eqn:v_1_v_2}
   v^1=\Big( \, 1, \, 0, \, -1, \, -2, \ldots, -(K_n+p-2) \,
   \Big)^T, \
   v^2=\Big( \, 0, \, 1, \, 2, \, 3, \ldots, K_n+p-1 \, \Big)^T.
\end{equation}
It is easy to see that $\Delta^2 v^k_j=0$ for $k=1, 2$ and all $j > 2$.
Hence $\pm v^k \in \Omega$ for $k=1, 2$, and it can be also
verified that $\sum^2_{k=1} v^k = \bkone$. Further, any
$b=(b_1,\ldots, b_{K_n+p})^T \in \Omega$ can be positively
generated as
\[
    b = \sum^2_{i=1}  \Big( \max(0, b_i) v^i + \max(0, -b_i) (-v^i)
    \Big)  + \sum^{K_n+p}_{i=3} \Delta^2 (b_i) v^i.
\]
By using these generators for $\Omega$,  we obtain the following
necessary and sufficient optimality conditions for an optimizer
$\hat b$:
\begin{equation} \label{eqn:optm_Dr}
   0 \, \le \, D_2 \hat b \, \perp \, \wt C \nabla g(\hat b) \, \ge \, 0 \ \ \ \mbox{ and } \ \ \
    \langle v^k, \, \nabla g(\hat b) \rangle =0, \ \ \forall \ k=1, 2,
\end{equation}
where $D_2 \in \mathbb R^{(K_n-2+p)\times (K_n+p)}$ is given by
\begin{equation} \label{eqn:D_2}
    D_2 \, = \, \left[\begin{array}{ccccccccc}
   1 & -2 & 1  &0& \cdots & 0&  0 & 0 & 0\\
   0 & 1 & -2 &1 & \cdots & 0&  0 & 0 & 0\\
    & \cdots & & & \cdots &  & \cdots & \cdots \\
   0 & 0 & 0 & 0 & \cdots &  1& -2 & 1 &0 \\
   0 & 0 &  0 &0 &\cdots &  0& 1 & -2 & 1
   \end{array} \right], \
\end{equation}
and  $\wt C \in \mathbb R^{(K_n-2+p)\times (K_n+p)}$ is given by
\begin{align*}
  &\wt C = \begin{bmatrix} \, v^{3} & \cdots & v^{K_n+p}\,
    \end{bmatrix}^T\\
      =& \left[\begin{array}{ccccccccc}
   0 & 0 & 1  & 2 & \cdots & \, (K_n+p-4) \, & \, (K_n+p-3) \, & \, (K_n+p -2) \\
   0 & 0 & 0 &1 & \cdots & \cdots & \, (K_n+p-4) \, & (K_n+p -3) \\
    & \cdots & & & \cdots & & \cdots & \cdots \\
   0 & 0 & 0 & 0 & \cdots & 0 & 1 & 2 \\
   0 & 0 &  0 &0 &\cdots & 0 & 0 & 1
   \end{array} \right].
\end{align*}
It can be shown via the definitions of $v^1$ and $v^2$ in
(\ref{eqn:v_1_v_2}) that the second optimality condition in
(\ref{eqn:optm_Dr}) can be equivalently written as
\[
  \sum^{K_n+p}_{i=1} \big( \nabla g(\hatb) \big)_i =0 \ \ \ \mbox{
     and } \ \ \ \sum^{K_n+p}_{i=1} (K_n+p-i+1) \big( \nabla g(\hatb) \big)_i
     =0,
\]
where $\nabla g(b) = \Lambda b - \bar y$. This gives rise to the
two boundary conditions. Moreover, noting that for any $k$, the
definitions of $v^1$ and $v^2$ in (\ref{eqn:v_1_v_2}) yield
\[
   \wt C_{k\bullet} \nabla g(\hat b) = \sum^k_{i=1} \sum^{i}_{j=1} \big( \nabla g(\hatb)
       \big)_j = (C^2)_{k\bullet} \nabla g(\hat b),
\]
we obtain the equivalent condition for the first optimality
condition in (\ref{eqn:optm_Dr}):
\begin{equation} \label{eqn:optimum_complementarity}
       0 \, \le \, D_2 \, \hatb \, \perp \, \big( C^2 \big)_{\gamma \bullet} \nabla g(\hatb) \, \ge \,
       0,
 \end{equation}
where $\gamma=\{ 1, \ldots, K_n+p-2\}$. By
$(C^2)_{\gamma\bullet} = C_{\gamma\bullet} C$, the proof is
complete.
\end{proof}

%
\subsection{Construction and Proof for Proposition~\ref{lem:PL_formulation}} \label{sect:proof_PL_form}

%

We first construct  certain equations that yield a linear selection function corresponding to the index set
index set $\alpha = \{ \, i \, | \, (D_2 \hat b)_i =0 \} \subseteq \{1, \ldots, K_n+p-2 \}$ ($\alpha$ may be empty). Specifically, for the given $\hat b$ and $\alpha$, we define a vector
$\wt b^\alpha$ and an associated family of index sets $\{
\beta^\alpha_i \}$ in the following steps:
\begin{itemize}
  \item [(1)] let $\ell_1 := \min_{3\le i \le K_n+p} \{ i \, :
   \, \Delta^2(\hat b_i) = 0 \}$, and $\ol\ell_1 := \max_{\ell_1 \le k
  \le K_n+p} \{ k \, : \, \Delta^2(\hat b_i) = 0, \ \forall i
  =\ell_1, \ldots, k \}$. Then inductively define, for $j \ge 1$,
  \begin{align*}
      \ell_{j+1} :=& \min_{1+\ol\ell_{j} \le i \le K_n} \{ i \, :
  \, \Delta^2(\hat b_i) = 0 \}, \\
       \ol\ell_{j+1} :=& \max_{\ell_{j+1} \le k \le K_n} \{ k \, : \, \Delta^2(\hat b_i) = 0,
       \ \ \forall i =
       \ell_{j+1}, \ldots, k \}.
  \end{align*}
  Suppose that we obtain $q$'s such $\ell_i, \ol\ell_i$,
   namely, $\ell_1, \ldots, \ell_q$ and $\ol\ell_1, \ldots,
  \ol\ell_q$. Define $\wh \beta^\alpha_{\ell_j} := \{ i \, : \,
  \ell_j -2 \le i \le \ol\ell_j\}$ for $j=1, \ldots, q$. Note that
  $|\wh \beta^\alpha_{\ell_j}| \ge 3$ for each $\ell_j$, and  for two consecutive index sets,
  $\ell_{j+1} \ge \ol
  \ell_j+2 $. Thus if the equality holds, then $\wh \beta^\alpha_{\ell_j} \cap \wh
  \beta^\alpha_{\ell_{j+1}}=\{ \ol\ell_j\}$; otherwise, the two consecutive index sets are disjoint.
 \item [(2)] let $\wh L:=K_n +p +q -  | \cup^q_{i=1} \wh \beta^\alpha_{\ell_i}|$, where $|\cdot|$ denotes
  the cardinality of an index set. For each $i \in \{ 1, \ldots, K_n+p\} \setminus
   \cup^q_{i=1} \wh\beta^\alpha_{\ell_j}$, define
   $\wh\beta^\alpha_{\ell_s} = \{ i \}$, where $s=(q+1), \ldots, \wh L$.
 \item [(3)] this step arranges the index sets
 $\wh\beta^\alpha_{\ell_j}$ in a monotone order as follows.
  For each $\wh\beta^\alpha_{\ell_i}$, let $\min (\wh\beta^\alpha_{\ell_i})$ denote the least element in
  $\wh\beta^\alpha_{\ell_i}$ (the similar notation will be used for $\max$
  below). Define $\ell_{s_1}
    :=\arg\min_{\ell_1, \ldots, \ell_{\wh L}} \{ \min (\wh\beta^\alpha_{\ell_i})
    \}$.
  Let $\wt \beta^\alpha_1 := \wh \beta^\alpha_{\ell_{s_1}}$.
  Then inductively define for each $j \ge 1$, $\wt \beta^\alpha_{j+1}
   :=\wh \beta^\alpha_{\ell_{s_{j+1}}}$, where
  \[
     \ell_{s_{j+1}} :=\arg\min_{ \{\ell_1, \ldots, \ell_{\wh L}\} \setminus \{ \ell_{s_1},
     \ldots, \ell_{s_j}\} } \{ \min (\wh\beta^\alpha_{\ell_i}) \}.
  \]
 \item [(4)] in this step, we regroup the index sets
 $\wh\beta^\alpha_{\ell_j}$ in a way that preserves desired structural
 properties to be used in the subsequent development.
%
%
  Define $p_0:=0$ and $$p_1 := \max\big( \, 1, \ \ \ \max\{ k \ge 1 \, : \,
    \wt\beta^\alpha_j \cap \wt\beta^\alpha_{j+1} \ne \emptyset,  \  \forall j=1,\ldots, k-1
    \} \, \big),$$
    and $ \beta^\alpha_1 : = \cup^{p_1}_{j=1}
    \wt\beta^\alpha_j$, the companion index set $\vartheta_1:=\{  \min(\wt\beta^\alpha_{j}), \forall j=1,\ldots,
    p_1 \} \cup\{ \max( \wt\beta^\alpha_{p_1} )   \}$. Recursively, define, for each $s \ge 1$,
   $$p_{s+1} :=\max\big( \, p_s+1, \ \ \ \max\{ k \ge p_s+1 \, : \,
    \wt\beta^\alpha_j \cap \wt\beta^\alpha_{j+1} \ne \emptyset,  \ \forall j=p_s+1,\ldots, k-1
    \} \, \big),$$
    and $ \beta^\alpha_{s+1} : = \cup^{p_{s+1}}_{j=p_s+1}
    \wt\beta^\alpha_j$, the companion index set $\vartheta_{s+1}:=\{  \min(\wt\beta^\alpha_{j}),
    \forall j=p_s+1,\ldots, p_{s+1} \} \cup\{ \max( \wt\beta^\alpha_{p_{s+1}} )
    \}$. Without loss of generality, we assume that the index
    elements of each $\vartheta_s$ are in the strictly increasing order. Hence, any two
    consecutive index sets in $\vartheta_s$ correspond to $\ell_j$
    and $\ol\ell_j$ defined in Step (1) with $\ell_{j+1}= \ol\ell_j$.
 \item [(5)] suppose that there are $L$ such the index sets $\vartheta_s$, and
 let $\vartheta :=\cup^L_{s=1} \vartheta_s$ whose index elements are in
 the strictly increasing order. Then $\wt\beta^\alpha := ( \ol \beta_i )$,
 where $i\in \vartheta$.
\end{itemize}
%
%
%

It is clear from the above construction that $\{ \beta^\alpha_i
\}$ forms a finite and disjoint partition of $\{ 1, \ldots,
K_n+p\}$, namely, $ \bigcup^L_{i=1} \, \beta^\alpha_{i}=\{ 1,
\ldots, K_n+p\}$ and $\beta^\alpha_{j} \cap \beta^\alpha_{k} =
\emptyset$ whenever $j \ne k$.

For a given index set $\alpha$, we drop the sign restriction (i.e., the
inequality $D_2 b \ge 0$) in (\ref{equ:opt-m-1}) and obtain its corresponding linear
selection function from the following (possibly redundant)
equations:
\begin{subequations} \label{eqn:linear_piece}
  \begin{eqnarray}
     (D_2 \hat b)_\alpha & = & 0, \label{eqn:linPiece_1} \\
     D_2 \hat b & \perp & C_{\gamma\bullet} C \big( \Lambda \hat b - \bar y\big), \label{eqn:linPiece_2} \\
     C_{\ol \alpha\bullet} C \big( \Lambda \hat b - \bar y\big) &
           = & 0, \label{eqn:linPiece_3} \\
     C_{(K_n+p) \bullet} \, \big( \Lambda \hat b - \bar y\big) & = &
C_{(K_n+p) \bullet} \,  C \, \big( \Lambda \hat b - \bar y\big)=0,
\label{eqn:linPiece_4}
  \end{eqnarray}
\end{subequations}
where $\ol\alpha := \{ 1, \ldots, K_n+p-2\} \setminus \alpha$.
Indeed, we shall use equations (\ref{eqn:linPiece_1}),
(\ref{eqn:linPiece_2}), and (\ref{eqn:linPiece_4}) to characterize
a linear piece.
Let $\wt b^\alpha$ denote the vector constituting
the free variables of equation (\ref{eqn:linPiece_1}). With this construction and notation, we are ready to prove Proposition~\ref{lem:PL_formulation} as follows.

\begin{proof}[Proof of Proposition~\ref{lem:PL_formulation}]
   We introduce some notation first. Let $m^\alpha_i:=|\beta^\alpha_i|$ and $h^\alpha_i
   :=   m^\alpha_i -1$, where $i=1, \ldots, L$. Note that
    if $m^\alpha_i > 1$, then $m^\alpha_i \ge 3$ such that $h^\alpha_i \ge 2$
     and $|\vartheta_i| \ge 2$.
   It follows from the definition of $\beta^\alpha_i$ that
   $\hat b^\alpha = (F_\alpha)^T \wt b^\alpha$, where the matrix
   \begin{equation} \label{eqn:F_alp}
     F_\alpha
   = \begin{bmatrix} F_{\alpha, 1} &  & & \\ & F_{\alpha, 2} & & \\ & & \ddots & \\& & & F_{\alpha, L}
   \end{bmatrix} \in \mathbb R^{\ell \times (K_n+p) }
   \end{equation}
    and each matrix
   block corresponding to $\beta^\alpha_k$ is given as follows: if
   $m^\alpha_k=1$, then $F_{\alpha, k}  = 1$; otherwise, assuming that the index elements in
   $\vartheta_k$ in Step (5) above are in the strictly increasing order without loss of generality, and letting
   $h^\alpha_{k, j} := \vartheta_k(j+1) - \vartheta_k(j) \ge 2$ for each $j=1, \ldots,
    |\vartheta_k|-1$, we determine $F_{\alpha, k}\in \mathbb R^{| \vartheta_k| \times m^\alpha_k}$
    from
   $\beta^\alpha_k$ constructed in Steps (1)-(5) as
    \begin{equation} \label{eqn:F_k}
  F_{\alpha, k}=\left[ \begin{array}{cccccccccccc}
  1 & \mathbf h^{\alpha, k, 1}          &   0&         0        &       0  &0    & \cdots&    \cdots &   \cdots&    \cdots              \\
  0 & \wt {\mathbf h}^{\alpha, k, 1}   &   1& {\mathbf h}^{\alpha, k, 2} &      0 &        0  & \cdots&  \cdots&    \cdots   &    \cdots          \\
  0 & 0                         &   0 & \wt {\mathbf h}^{\alpha, k, 2}&1 & {\mathbf h}^{\alpha, k, 3}& \cdots &\\
    & \cdots                    &    \cdots                          &   &&&\\
    & \cdots                    &    \cdots                          &   &&&\\
    &                    &                              &   &&\cdots&\cdots \\
    &                    &                              &   &&\cdots&\cdots \\
    & \cdots                    &                              &  \cdots &  & &  &\cdots & {\mathbf h}_{\alpha, k, w_k} & 0\\
    & \cdots                    &                              &  \cdots &   & &  & \cdots& \wt {\mathbf h}^{\alpha, k, w_k} & 1
    \end{array}
  \right],
  \end{equation}
   where $w_k :=|\vartheta_k|-1$, and the row vectors
   \begin{eqnarray*}
   {\mathbf h}^{\alpha,k,j}  &: =& \begin{bmatrix} \frac{h^\alpha_{k,j} -1}{h^\alpha_{k,j} } &
       \frac{h^\alpha_{k,j}-2}{h^\alpha_{k,j}} & \cdots &  \frac{1}{h^\alpha_{k, j}}\end{bmatrix},  \quad j=1, \ldots, w_k, \\
       \wt {\mathbf h}^{\alpha,k,j} &: = &  \begin{bmatrix} \frac{1}{h^\alpha_{k, j}} & \frac{2}{h^\alpha_{k, j} } & \cdots &        \frac{h^\alpha_{k,j} -1}{h^\alpha_{k,j}  } \end{bmatrix}, \quad j =1, \ldots, w_k.
   \end{eqnarray*}
  %
%
%

  For notational simplicity, let $v:= \Lambda \hat b - \bar y$. In view of the complementarity condition
  in (\ref{equ:opt-m-1}), we have $(D_2 \hat b)^T
  C_{\gamma\bullet} C v =0$. Since $\hat b = (F_\alpha)^T \wt
  b^\alpha$,  $(\wt b^\alpha)^T F_\alpha ( D^T_2 C_{\gamma\bullet} C
  v)=0$. Moreover, it can be further verified that
  \[
       D^T_2 C_{\gamma\bullet} C = \begin{bmatrix} I_{K_n+p-2} & &
       0_{(K_n+p-2)\times 2}
       \\ E &  & 0_{2\times 2} \end{bmatrix} \in \mathbb R^{(K_n+p)\times (K_n+p)},
  \]
  where
  \[
     E = \begin{bmatrix} -(K_n+p-1) & -(K_n+p-2) & \cdots & \cdots & -2 \\
          K_n+p-2 & K_n+p-3 & \cdots & \cdots & 1 \end{bmatrix} \in
          \mathbb R^{2\times (K_n+p-2)}.
  \]
  It also follows from the boundary conditions $C_{(K_n+p)\bullet} v
  = C_{(K_n+p)\bullet} C v = 0$ and elementary row operations that $[
     -E \ \ I_2 ] v =0$. Therefore, we obtain $D^T_2 C_{\gamma\bullet} C v = I_{(K_n+p)} ~v = v$.
   Hence,  $(\wt b^\alpha)^T F_\alpha ( D^T_2 C_{\gamma\bullet} C
  v)= (\wt b^\alpha)^T  F_\alpha v =0$. Recall that for the given index set $\alpha$,
  $\wt b^\alpha$ corresponds to the free variables of the equation
  (\ref{eqn:linPiece_1}). Hence, $\wt b^\alpha$ is arbitrary such that $F_\alpha v =0$. This leads to
  \[
    F_\alpha \, \Lambda (F_\alpha)^T \, \wt b^{\, \alpha} = F_\alpha \, \bar y.
  \]
   Letting $\wt \Lambda^\alpha =F_\alpha \, \Lambda (F_\alpha)^T$ and
   $\wt y^{\, \alpha} = F_\alpha \, \bar y$, we
  obtain the linear equation for $\wt b^{\, \alpha}$. Since $F_\alpha$ is of
  full row rank and $\Lambda$ is positive definite, $\wt \Lambda^\alpha $ is  positive
  definite and hence is invertible. Consequently, we have $\hat b^\alpha (\bar y) = F^T_\alpha \wt b^\alpha(\bar y) = F^T_\alpha \big(F_\alpha \Lambda F^T_\alpha \big)^{-1} F_\alpha \bar y$.
\end{proof}

%

\section{Proof of Theorem ~\ref{theorem:uniform_Lip}}
\label{sect:proof_uniform_Lip}

We divide the proof of Theorem~\ref{theorem:uniform_Lip} into
several steps. We first establish a result pertaining to $F_\alpha
F^T_\alpha$.

\begin{lemma} \label{lem:F*Ftranspose}
 For any $K_n$ and $\alpha$, $F_\alpha F^T_\alpha$ is a strictly
 diagonally dominant, nonnegative, tridiagonal matrix.
\end{lemma}

\begin{proof}
  Recall that $\ell:=K_n+p- |\alpha|$. For notational simplicity,
  let $G:=F_\alpha F^T_\alpha$.
  First of all, it is easy to verify via (\ref{eqn:F_alp}) and
   (\ref{eqn:F_k}) that $G$ is
   the  $\ell\times \ell$ tri-diagonal matrix given by
  \[
      \begin{bmatrix} d_{11} & \wt \eta_1 & 0 & \cdots  & \cdots & 0 \\
        \wt \eta_1 & d_{22} &  \wt\eta_2 &  & &  \\
         & \wt \eta_2 & d_{33} &  \wt\eta_3 & &   \\
       &  & \ddots & \ddots & \ddots & \\
       & & &  \wt \eta_{\ell-2} & d_{(\ell-1)(\ell-1)} & \wt \eta_{\ell-1} \\
       0 & \cdots  & \cdots & 0  & \wt \eta_{\ell-1} & d_{\ell \ell}
       \end{bmatrix}.
    \]
   The entries on the three diagonal bands are determined as follows. Consider
    $F_\alpha$ in (\ref{eqn:F_alp}) with $L$ blocks.
    Fix $k\in \{1, \ldots, L\}$. If $m^\alpha_k=1$, then $F_{\alpha, k} F^T_{\alpha, k}$
  is a real number that appears on the diagonal of $G$. Denoting this
  number by $d_{ss}$, we have $
    d_{ss} = F_{\alpha, k} F^T_{\alpha, k}  = 1 $ and
    $G_{s(s+1)}=G_{(s+1)s}=0$, $G_{s j} =0$ for
all $j \le s-2$ and $j \ge s+2$. If $m^\alpha_k>1$, then
$F_{\alpha, k} F^T_{\alpha, k} $ is a symmetric, positive definite
matrix of order $|\vartheta_k|$ that forms a diagonal block of
$G$. Making use of the structure of $F_{\alpha, k}$ given in the proof of
Proposition~\ref{lem:PL_formulation} and somewhat lengthy computation, we obtain the
following results in two separate cases (recalling
$w_k:=|\vartheta_k|-1$):
\begin{itemize}
  \item [(1)] $k=1$ or $k=L$. For $k=1$,
  \begin{eqnarray*}
      d_{11} & =&  1+
          \frac{ (h^\alpha_{1,1}-1)(2h^\alpha_{1,1}-1)}{6 h^\alpha_{1,1} }, \\ 
      \wt \eta_{s} & =  &  G_{s(s+1)} \, = \, G_{(s+1)s} \, = \,
           \frac{ (h^\alpha_{1, s})^2-1 }
           {6 h^\alpha_{1, s}}, \ \ \forall \ s=1, \ldots, w_1, \\
      d_{ss} & =  &  \frac{2 (h^\alpha_{1,s-1})^2+1 }{6 h^\alpha_{1, s-1} }
        + \frac{ 2 (h^\alpha_{1,s})^2+1 }{6 h^\alpha_{1, s} },
          \qquad \  \forall \ s=2, \ldots, w_1,  \\
      d_{(w_1+1)(w_1+1)} & = & \frac{ (h^\alpha_{1, w_1} +1)(2h^\alpha_{1, w_1}+1) }
      {6 h^\alpha_{1, w_1} }. 
%
  \end{eqnarray*}
  Besides, $G_{(w_1+1)(w_1+2)}=G_{(w_1+2)(w_1+1)}=0$ and
  for each $s=1, \ldots, w_t$, $G_{s j}=0, \forall j \ge s+2$ and $j \le s-2$.
  For $k=L$, the similar
   results can be established by using the symmetry of the rows of $F_{\alpha, L}$.
 \item [(2)] $k\in \{ 2, \ldots, L-1\}$. In this case, suppose
 that the $(1,1)$-element of $F_{\alpha, k} F^T_{\alpha,
 k}$, which is a diagonal entry of $G$, is denoted by
 $d_{tt}$. Then we have
   \begin{eqnarray*}
     d_{tt} & = & 1 +
          \frac{ (h^\alpha_{k,1}-1)(2h^\alpha_{k,1}-1)}{6 h^\alpha_{k,1}  },  \label{eqn:dss} \\
     \wt \eta_{t+s} & = & G_{(t+s)(t+s+1)} \, = \,
             \frac{ (h^\alpha_{k, s})^2-1 }
           {6 h^\alpha_{k, s}}, \quad \ \ \ \forall \ s=1, \ldots, w_k, \\
    d_{(t+s)(t+s)} & = & \frac{2 (h^\alpha_{k,s+1})^2+1 }{6
      h^\alpha_{k, s+1} }
        + \frac{ 2 (h^\alpha_{k,s})^2+1 }{6 h^\alpha_{k, s} }, \quad \  \forall \ s=1, \ldots, w_k-1,  \\
    d_{(t+w_k)(t+w_k)} & = &  \frac{ (h^\alpha_{k, w_k} +1)(2h^\alpha_{k, w_k}+1) }
      {6 h^\alpha_{k, w_k} }. \label{eqn:ds+1}
   \end{eqnarray*}
    In addition, for each $s=t, \ldots, t+w_k+1$,
    $G_{s j}=0$ for all $j \le s-2$ and $j \ge s+2$,  and
    $G_{t(t-1)}=G_{(t+w_k+1)(t+w_k+2)}=0$.
\end{itemize}
 Due to  $G_{t(t-1)}=0$ and the symmetry of $G$, we further deduce that if a diagonal entry
 $d_{tt}= G_{tt}$ with $t \ge 2$
 corresponds to a scalar $ F_{\alpha, k} F^T_{\alpha,
 k}$ (i.e., $m^\alpha_k=1$), then $G_{(t-1)t} = 0$.
 (Recall that $G_{t(t+1)} = 0$ has been shown before.)
 Similarly, if $d_{tt}$ is the first diagonal entry of
  $ F_{\alpha, k} F^T_{\alpha,
 k}$, then $G_{(t-1)t} = 0$.

In the next, we show that $G$ is strictly diagonally dominant.
For a given $G \in \mathbb R^{\ell\times \ell}$, define
\[
  \xi_1 := d_{11}-|\wt \eta_1|, \ \  \xi_\ell:= d_\ell-|\wt
  \eta_{\ell-1}|,  \quad \mbox{ and } \ \
   \xi_i := d_{ii} - |\wt \eta_{i-1}|-|\wt
  \eta_{i}|, \ \ \ \forall \, i\in\{2,\ldots, \ell-1\}.
\]
  In light of the entries of $G$ obtained above, we have, for each
  $k\in\{1, \ldots, L\}$,
  \begin{itemize}
      \item [(1.1)] if $m^\alpha_k=1$, then
       $\xi_i= 1$.
      \item [(1.2)] if $m^\alpha_k>1$ with $k=1$, then (i) the corresponding
       $\xi_i=d_{11} - |\eta| \ge \frac{1}{2}+\frac{h^\alpha_{1,1}}{6} $;
       (ii) for $s=2, \ldots, w_1$, the corresponding $\xi_i=d_{ss}-|G_{s(s-1)}|
        -|G_{s(s+1)}| \ge
        \big( h^\alpha_{1,s-1} + h^\alpha_{1,s}\big)/6$; and (iii) the corresponding
       $\xi_i=d_{(w_1+1)(w_1+1)} - |G_{(w_1+1)w_1}| -|G_{(w_1+1)(w_1+2)}| \ge
        \frac{1}{2}+\frac{h^\alpha_{1,w_1}}{6}$.
        The similar results can be obtained for $m^\alpha_k>1$ with
        $k=L$ using symmetry.
      \item [(1.3)] if $m^\alpha_k>1$ with $k\in\{ 2, \ldots,
      L-1\}$, then (i) the corresponding
       $\xi_i=d_{tt} - |G_{t(t-1)}| -|G_{t(t+1)}| \ge
       \frac{1}{2}+\frac{h^\alpha_{k,1}}{6}$;
       (ii) for $s=1, \ldots, w_k-1$, the corresponding $\xi_i=d_{(t+s)(t+s)}-|G_{(t+s)(t+s-1)}|
        -|G_{(t+s)(t+s+1)}| \ge
        \big( h^\alpha_{k,s} + h^\alpha_{k,s+1}\big)/6$; and (iii) the corresponding
       $\xi_i=d_{(t+w_k)(t+w_k)} - |G_{(t+w_k)(t+w_k-1)}| -|G_{(t+w_k)(t+w_k+1)}| \ge
        \frac{1}{2}+\frac{h^\alpha_{k,w_k}}{6}$.
  \end{itemize}
   Consequently, $\xi_i >0$ for all $i$ and $G$ is strictly diagonally
   dominant.
\end{proof}


For the given $F_\alpha$, define $\eta_i$ as the sum of the
entries in the $i$th row of $F_\alpha$, $i=1, \ldots, \ell$. We
have:
\begin{itemize}
      \item [(i)] if $m^\alpha_k=1$, then
       $\eta_i= 1$.
      \item [(ii)] if $m^\alpha_k>1$ with $k=1$, then (i) for $s=1$, the corresponding
       $\eta_i=\frac{1+h^\alpha_{1,1}}{2} $;
       (ii) for $s=2, \ldots, w_1$, the corresponding $\eta_i=
        \big( h^\alpha_{1,s-1} + h^\alpha_{1,s}\big)/2$; and (iii) for $s=w_1+1$, the corresponding
       $\eta_i=\frac{1+h^\alpha_{1,w_1}}{2}$.
        The similar results can be obtained for $m^\alpha_k>1$ with
        $k=L$ using symmetry.
      \item [(iii)] if $m^\alpha_k>1$ with $k\in\{ 2, \ldots,
      L-1\}$, then (i) the corresponding
       $\eta_i =\frac{1+h^\alpha_{k,1}}{2}$;
       (ii) for $s=1, \ldots, w_k-1$, the corresponding $\eta_i=
        \big( h^\alpha_{k,s} + h^\alpha_{k,s+1}\big)/2$; and (iii) the corresponding
       $\eta_i= \frac{1+h^\alpha_{k,w_k}}{2}$.
  \end{itemize}
Hence, each $\eta_i>0$. Define the diagonal matrix
\begin{equation} \label{eqn:Xi}
 \Xi_\alpha : = \mbox{diag} \Big ( \eta^{-1}_1, \ldots, \eta^{-1}_\ell \Big ).
\end{equation}
%
%

The next lemma shows the equivalence of the (absolute) row sum of
$F_\alpha F^T_\alpha$ and that of $F_\alpha$.

\begin{lemma} \label{lem:row_sum}
  For any $K_n$ and $\alpha$, $\eta_j =
  \sum^\ell_{k=1}(F_\alpha F^T_\alpha)_{jk}$ for each $j=1, \ldots, \ell$.
\end{lemma}

\begin{proof}
  Let $(F_{\alpha})_{j\bullet}$ denote the $j$th row of $F_{\alpha}$.
  Then
  \[
    \sum^\ell_{k=1} | (F_\alpha F^T_\alpha)_{jk}| = \langle
      (F_{\alpha})_{j\bullet}, \sum^\ell_{k=1} \big( (F_\alpha)_{k\bullet}\big)^T \rangle
     = \langle  (F_\alpha)_{j\bullet}, {\bf 1} \rangle =
     \sum^\ell_{k=1}(F_\alpha)_{jk} = \eta_j ,
  \]
  where we use the fact that the sum of all rows of $F_\alpha$ is
  $\mathbf 1:=(1, \ldots, 1)$.
\end{proof}

\begin{proposition} \label{prop:F_eign}
 For any $K_n$ and $\alpha$, the following statements hold:
 \begin{itemize}
    \item [(1)] the eigenvalues of $\Xi_\alpha F_\alpha \Lambda F^T_\alpha$ and $\Xi_\alpha F_\alpha F^T_\alpha$ are all positive reals;
    \item [(2)] $\displaystyle   \lambda_{\min}(\Lambda) \lambda_{\min}( \Xi_\alpha F_\alpha  F^T_\alpha )
    \le \lambda_{\min} ( \Xi_\alpha F_\alpha \Lambda F^T_\alpha
   )$ and \\ $ \lambda_{\max} ( \Xi_\alpha F_\alpha \Lambda F^T_\alpha)
    \le \lambda_{\max}(\Lambda) \lambda_{\max}( \Xi_\alpha F_\alpha  F^T_\alpha
    )$.
 %
 \end{itemize}
\end{proposition}

\begin{proof}
  (1) For the diagonal matrix $\Xi_\alpha$, define $\Xi^{1/2}_\alpha := \mbox{diag}\big(\sqrt{\eta^{-1}_1}, \ldots,
\sqrt{\eta^{-1}_\ell} \big)$. Let $\sigma(A)$ denote the spectrum of a
square matrix $A$, i.e., the collection of all eigenvalues of $A$.
We thus have
\begin{eqnarray*}
 &&\lambda'\in \sigma(\Xi^{1/2}_\alpha F_\alpha \Lambda F^T_\alpha
   \Xi^{1/2}_\alpha )\\
    & \Longleftrightarrow  & \det \big(\lambda' I - \Xi^{1/2}_\alpha F_\alpha \Lambda F^T_\alpha
   \Xi^{1/2}_\alpha \big)=0 \\
   & \Longleftrightarrow & \det  ( \Xi^{1/2}_\alpha) \cdot \det \big(\Xi^{-1/2}_\alpha \cdot \lambda' \cdot \Xi^{-1/2}_\alpha  - F_\alpha \Lambda F^T_\alpha
  \big ) \cdot \det(\Xi^{1/2}_\alpha)=0 \\
   & \Longleftrightarrow & \det( \lambda' \cdot \Xi^{-1}_\alpha  - F_\alpha \Lambda F^T_\alpha
   ) = 0 \\
  &  \Longleftrightarrow &  \det( \lambda' I  - \Xi_\alpha F_\alpha \Lambda F^T_\alpha
   ) = 0 \\
   & \Longleftrightarrow & \lambda'\in \sigma( \Xi_\alpha F_\alpha \Lambda
   F^T_\alpha).
 \end{eqnarray*}
 Since $\Lambda$ is positive definite and $F_\alpha$ is row
 linearly independent, $\Xi^{1/2}_\alpha F_\alpha \Lambda F^T_\alpha
   \Xi^{1/2}_\alpha$ is positive definite such that all the
   eigenvalues of $\Xi_\alpha F_\alpha \Lambda
   F^T_\alpha$ are positive reals. By replacing $\Lambda$ by the
   identity matrix, we see that the same holds for the eigenvalues
   of $\Xi_\alpha F_\alpha F^T_\alpha $.

   (2) By Statement (1),  $\lambda_{\min} (\Xi_\alpha F_\alpha \Lambda
   F^T_\alpha  ) =\lambda_{\min}(\Xi^{1/2}_\alpha F_\alpha
\Lambda F^T_\alpha \Xi^{1/2}_\alpha )$ and $\lambda_{\max}
(\Xi_\alpha F_\alpha \Lambda  F^T_\alpha ) =
\lambda_{\max}(\Xi^{1/2}_\alpha F_\alpha \Lambda F^T_\alpha
\Xi^{1/2}_\alpha)$. Further, for any $x\ne 0$,
\[
  \frac{ x^T \Xi^{1/2}_\alpha F_\alpha \Lambda F^T_\alpha \Xi^{1/2}_\alpha x }{ x^T x }
  = \frac{ x^T \Xi^{1/2}_\alpha F_\alpha \Lambda F^T_\alpha \Xi^{1/2}_\alpha x }{ x^T \Xi^{1/2} F_\alpha  F^T_\alpha \Xi^{1/2} x }
     \cdot \frac{x^T \Xi^{1/2}_\alpha F_\alpha  F^T_\alpha \Xi^{1/2}_\alpha x
     }{ x^T x}.
\]
Therefore, using the fact that all the eigenvalues of $\Lambda$
are positive, we obtain
\begin{align*}
   \lambda_{\min}(\Lambda) \lambda_{\min}( \Xi^{1/2} F_\alpha  F^T_\alpha \Xi^{1/2} )
   & \le \lambda_{\min} ( \Xi^{1/2}_\alpha F_\alpha \Lambda F^T_\alpha
   \Xi^{1/2}_\alpha ) \le \lambda_{\max} ( \Xi^{1/2}_\alpha F_\alpha \Lambda F^T_\alpha
   \Xi^{1/2}_\alpha ) \\
   & \le \lambda_{\max}(\Lambda) \lambda_{\max}( \Xi^{1/2}_\alpha F_\alpha
   F^T_\alpha \Xi^{1/2}_\alpha )
\end{align*}
Since $\lambda_{\min}( \Xi^{1/2}_\alpha F_\alpha F^T_\alpha
\Xi^{1/2}_\alpha ) = \lambda_{\min}( \Xi_\alpha F_\alpha
F^T_\alpha)$ and $\lambda_{\max}( \Xi^{1/2}_\alpha F_\alpha
F^T_\alpha \Xi^{1/2}_\alpha ) = \lambda_{\max}( \Xi_\alpha
F_\alpha F^T_\alpha)$, the desired inequalities follow.
\end{proof}

The following proposition attains uniform upper and lower bounds for
the eigenvalues of $\Xi_\alpha F_\alpha F^T_\alpha$, regardless of
$K_n$ and $\alpha$.

\begin{proposition} \label{prop:F_eign_bound}
  For any $K_n$ and $\alpha$, $$ 1/3\le \lambda_{\min} ( \Xi_\alpha F_\alpha F^T_\alpha
   ) \le \lambda_{\max} ( \Xi_\alpha F_\alpha F^T_\alpha
    ) \le 1.$$
\end{proposition}

\begin{proof}
 (1) Uniform upper bound.
%
%
  By Lemma~\ref{lem:row_sum},
  we have $\sum^\ell_{k=1} | (F_\alpha F^T_\alpha)_{jk}| = \sum^\ell_{k=1} G_{jk} = \eta_j
  $.
  Hence, $\sum^\ell_{k=1} (\Xi_\alpha F_\alpha
  F^T_\alpha)_{jk}=1$ for all $j=1, \ldots, \ell$. It follows from
  \cite[Corollary 6.1.5]{HornJohson_book85} that $\lambda_{\max} ( \Xi_\alpha F_\alpha F^T_\alpha
    ) \le 1$.

(2) Uniform lower bound. To establish this bound, we exploit
Gersgorin's Disc Theorem, say \cite[Theorem
6.1.1]{HornJohson_book85}. Notice that each $\xi_i$ defined in
Lemma~\ref{lem:F*Ftranspose} is the difference between the $i$th
diagonal of $G:=F_\alpha F^T_\alpha$ and the deleted absolute
 row sum of the $i$th row of $G$. Hence, by  Gersgorin's Disc Theorem, we see
that $\lambda_{\min} (\Xi_\alpha F_\alpha F^T_\alpha) \ge \min_{i}
(\xi_i/\eta_i)$. Further, using the lower bound of $\xi_i$ given
in Lemma~\ref{lem:F*Ftranspose} and the equality for $\eta_i$
given before Lemma~\ref{lem:row_sum}, it is easy to verify
$\xi_i/\eta_i \ge 1/3$ for all $i$, $\alpha$ and $K_n$. This
yields the desired uniform lower bound.
\end{proof}

%
%

The next result establishes a uniform bound on the
$\ell_\infty$-norm of $F^T_\alpha (F_\alpha \Lambda
F^T_\alpha)^{-1} F_\alpha $.

\begin{proposition} \label{prop:Linf_norm}
 There exists $c_{\infty, p}>0$ (dependent on $p$ only) such that for any $K_n$ and $\alpha$,
   $ \| F^T_\alpha (F_\alpha \Lambda F^T_\alpha)^{-1}  F_\alpha \|_\infty \le c_{\infty, p}$.
\end{proposition}

\begin{proof}
  For any $K_n$ and any $\alpha$, let $\Xi_\alpha$ be that defined
  in (\ref{eqn:Xi}). Then
  \begin{align*}
   \| F^T_\alpha (F_\alpha \Lambda F^T_\alpha)^{-1} F_\alpha\|_\infty &= \| F^T_\alpha \cdot
   (\Xi_\alpha F_\alpha \Lambda F^T_\alpha)^{-1} \cdot (\Xi_\alpha F_\alpha)\|_\infty \\
   &\le \| F^T_\alpha\|_\infty \cdot
     \| (\Xi_\alpha F_\alpha \Lambda F^T_\alpha)^{-1}\|_\infty \cdot \|\Xi_\alpha F_\alpha
    \|_\infty.
   \end{align*}
    It is easy to verify $\| F^T_\alpha\|_\infty=1$. Furthermore,
    due to the definition of $\Xi_\alpha$, we have $\|\Xi_\alpha F_\alpha
    \|_\infty=1$. In what follows, we show that $\| (\Xi_\alpha F_\alpha \Lambda
    F^T_\alpha)^{-1}\|_\infty$ is uniformly bounded using the
    banded structure of  the matrix and other technical results developed
    before. This will give rise to a uniform bound.

    Let $F_\alpha$ be of $\ell$ rows. We consider two cases as
    follows:

   (i) $\ell \ge 2 (p+1)$. In this case,
    by the structure of $F_\alpha$ shown in (\ref{eqn:F_alp}), we see
    via straightforward computation that $H:=\Xi_\alpha F_\alpha \Lambda
    F^T_\alpha $ is a banded symmetric matrix with bandwidth $p$, i.e., $(H)_{ij}=0$ whenever
    $|i-j|>p$.
    It is known from \cite[Lemma 6.2]{zhou_98} that for a
fixed spline degree $p$, there exist positive constants
${\ul\mu}_p$ and $\ol\mu_p$ (dependent on $p$ only) such that $\ul\mu_p \le
\lambda_{\min}(\Lambda) \le \lambda_{\max}(\Lambda) \le \ol\mu_p$
for any $K_n$.
    It  thus follows from Propositions~\ref{prop:F_eign} and \ref{prop:F_eign_bound} that
    $\|H \|_2 =\lambda_{\max}(H)\le \ol\mu_p$, where $\ol\mu_p$ is
    independent of $K_n$ and $\alpha$. Similarly, $\| H^{-1} \|_2
    \le 1/\lambda_{\min}(H) \le 3/\ul\mu_p$. Hence, for $F_\alpha$ with $\ell\ge 2(p+1)$,
    it follows from \cite[Theorem 2.2]{Demko_siam77} that there exists
    $c'>0$ (independent of $K_n$ and $\alpha$) such that $\|(H^{-1})_{i\bullet}\|_1 \le
    c'$ for all $i=1, \ldots, \ell$, where
    $(H^{-1})_{i\bullet}$ denotes the $i$th row of $H^{-1}$. In
    other words, $\| H^{-1}\|_\infty \le c'$.

    (ii) $\ell<2(p+1)$.
    For any $F_\alpha$ in this case, we introduce the block diagonal matrix
    $\wt H:=\mbox{diag}( H, 1, \ldots, 1)$ such that $H'$ has $2(p+1)$
    rows. Hence, $\wt H$ is a banded symmetric matrix with  bandwidth $p$
    and satisfies
     $\| \wt H \|_2 \le \max(\ol\mu_p, 1)$, $\|\wt
    H^{-1}\|_2 \le \max(3/\ul\mu_p, 1)$. Thus there exists $c''>0$ (independent of $K_n$ and $\alpha$)
    such that $\| H^{-1} \|_\infty \le \| \wt H^{-1}  \|_\infty \le c''$.

    Consequently,
    $c_{\infty, p}:=\max(c', c'')$ is the desired uniform bound with respect to the $\ell_\infty$-norm.
\end{proof}

Along with the above results, we finally complete the proof of the
uniform Lipschitz property below.

\begin{proof}[Proof of Theorem~\ref{theorem:uniform_Lip}]
   The uniform bound on   $\| F^T_\alpha (F_\alpha \Lambda F^T_\alpha)^{-1}
    F_\alpha \|_\infty $ has been established in 
    Proposition \ref{prop:Linf_norm}. The second statement follows directly
     from the continuous and piecewise
    linear property of $\hat b$ and polyhedral theory \cite[Proposition 4.2.2]{FPang03_book}.
\end{proof}


\section{Proof of Theorem \ref{thm:convergence}}

We introduce some notation first.
Let $\bar f^{[p]}$ be the spline estimator based on noise free data, i.e.,  $\bar f^{[p]}(x) = \sum_{k=1}^{K_n+p} \bar b_k B_k^{[p]}(x)$, where
\begin{equation} \label{eqn:optimal_b_y}
  \bar b \, := \, \arg\min_{b\in \Omega} {1\over 2} b^T \Lambda b - b^T \mathbb E(\bar y).
\end{equation}

%
%

Propositions \ref{lem:bias} and \ref{lem:random}  below give rise to uniform bounds for the bias and stochastic terms of estimation error in the sup-norm, respectively.

\begin{proposition} \label{lem:bias}
If $1\le r'\le r$, there exists a constant $C_{1r'}$, which depends on $r'$ only, such that
\begin{equation}\label{equ:bias}
\sup_{f\in {\cal C}(r, L)}\|\bar f_{(r')} - f\|_\infty \le C_{1r'}\cdot L \cdot K_n^{-r}.
\end{equation}
In particular, if $r=2$, then $C_{1 r'}$ is independent of $r'$.
\end{proposition}

\begin{proof}
Consider the case when $1\le r\le 2$ first. Hence $\lceil r'-1\rceil = 1$. Let $\tilde f$ be a piecewise linear function such that $\tilde f(\kappa_k)=f(\kappa_k)$. For any $x\in [\kappa_{k-1}, \kappa_k]$, $k=1, \ldots, K_n$, there exist $\xi_x, \tilde \xi_x \in (\kappa_{k-1}, \kappa_k)$ such that
\begin{align*}
&\tilde f(x) - f(x) \\
=& \tilde f(\kappa_{k-1}) + K_n \big (\tilde f(\kappa_k) - \tilde f(\kappa_{k-1}) \big)(x-\kappa_{k-1}) - \big[ f(\kappa_{k-1})+f'(\xi_x)(x-\kappa_{k-1})\big] \\
=& \big[ f'(\tilde\xi_x) - f'(\xi_x) \big] (x-\kappa_{k-1})\le  L |\tilde\xi_x -\xi_x |^{\ell} |x-\kappa_{k-1}| \le L K_n^{-r}.
\end{align*}
Thus $\|\tilde f - f\|_\infty \le L K_n^{-r}$.
Let  $\vec f:=(f(x_1), \ldots, f(x_n))^T$, $\vec{\tilde f}:=(\tilde f(x_1), \ldots, \tilde f(x_n))^T$, and let
$\tilde b$ be the optimal solution of (\ref{eqn:optimal_b_y}) with $\mathbb E(\bar y)$ replaced by $X^T\vec{\tilde f} /\beta_n$. Since $\tilde f$ is a piecewise linear and convex function, we have $\tilde b =(\tilde f(\kappa_1), \ldots, \tilde f(\kappa_{K_n}) )^T$. It follows from Theorem ~\ref{theorem:uniform_Lip} that
\begin{eqnarray*}
\|\bar f^{[1]} - \tilde f\|_\infty & \le & \|\bar b - \tilde b\|_\infty \le {c_{\infty, 1}\over\beta_n} \|X^T(\vec f - \vec{\tilde f})\|_\infty \\
&\le & { c_{\infty,1}\over\beta_n} \|X^T\|_\infty \|f - \tilde f\|_\infty = c_{\infty,1} \varrho \|f - \tilde f\|_\infty,
\end{eqnarray*}
where $\|X^T\|_\infty = \sum_{i=1}^n B_{2}^{[1]}(x_i)$ and $\varrho: =\sum_{i=1}^n B_{2}^{[1]}(x_i)/\sum_{i=1}^n B_{2}^{[1]}(x_i)^2$.
Therefore, letting $C_{1 r'}:=c_{\infty, 1} \varrho$ which is independent  of $r'$, we have
\[
\|\bar f ^{[1]}- f\|_\infty \le (1+c_{\infty,1} \varrho)\|f - \tilde f\|_\infty \le (1+c_{\infty,1} \varrho) L K_n^{-r} = C_{1 r'} L K^{-r}_n.
\]
Next, consider the case when $r>2$. If $r'\le 2$, a similar argument as above yields (\ref{equ:bias}). If $r'>2$, it is shown in Theorem \ref{thm:uncon} that an unconstrained estimator and the constrained one are asymptotically equivalent. Since (\ref{equ:bias}) holds for the unconstrained estimator \cite{zhou_98}, the proof is complete.
\end{proof}

\begin{proposition} \label{lem:random}
There exists a positive constant $C_{2r}$, which depends on $r$ only, such that for any $u>0$,
\begin{align}
P\Big( \|\hat f_{(r)} - \bar f_{(r)}\|_\infty \ge u \Big) \le (K_n+p)\exp\Big\{ -{n\over 2 K_n C_{2r}^2\sigma^2}u^2 \Big\}.
\end{align}
In particular, if $r\in [1, 2]$, then $C_{2 r}$ is independent of $r$.
\end{proposition}

\begin{proof}
Recall $p =\lceil r-1\rceil$. By Theorem \ref{theorem:uniform_Lip} and (\ref{eqn:beta_n}), we have
%
%
%
$$\|\hat f_{(r)} - \bar f_{(r)}\|_\infty \le {\sigma c_{\infty,p}\over \sqrt{\beta_n}}\sup_{k=1, \ldots, K_n+p} |\xi_k| = {\sigma c_{\infty,p}\over \sqrt{C_{\beta,p}}}\sqrt{K_n\over n}\sup_{k=1, \ldots, K_n+p} |\xi_k|,$$
where $\xi_k = \sum_{i=1}^n B_k^{[p]}(x_i)\epsilon_i/\sqrt{\beta_n}$. Letting $C_{2r}: = c_{\infty, p}/\sqrt{C_{\beta,p}}$ which is dependent on $r$ only (but independent of $r$ if $r\in[1,2]$), we have
$$
\displaystyle \|\hat f_{(r)} - \bar f_{(r)}\|_\infty \, \le \,  \sigma  C_{2r}\sqrt{K_n\over n}~ \tilde \Gamma_r,
$$
where $ \tilde\Gamma_r = \max_{k=1, \ldots, K_{n+p}}|\xi_k|$. Hence,
by using the implication:
 $ Z\sim N(0, 1) \Longrightarrow P(Z>t)\le {1\over 2}e^{-t^2/2}, \forall \ t \ge 0$, we have
\begin{align*}
& P\Big( \|\hat f_{(r)} - \bar f_{(r)}\|_\infty \ge u \Big)   \le P\Big( \tilde\Gamma_r \ge {u\over C_{2r}\sigma}\sqrt{n\over K_n}    \Big)\\
\le & (K_n+p) P\Big\{|\xi_k|\ge { u\over C_{2r}\sigma}\sqrt{n\over K_n}  \Big\}\le(K_n+p) \exp\Big\{-{n\over 2 K_n C_{2r}^2\sigma^2}u^2 \Big\}.
\end{align*}
\end{proof}

Let $$\displaystyle T_n := C_{2r}\sigma \sqrt{2\over 2r+1} \sqrt{\log n\over n} K_n^{1\over 2}.$$ It follows from Proposition \ref{lem:random} that,
\begin{align*}
\mathbb E\Big(\|\hat f_{(r)} - \bar f_{(r)}\|_\infty\Big)&\le T_n + \int_{T_n}^\infty P\Big( \|\hat f_{(r)} - \bar f_{(r)}\|_\infty>t   \Big)dt\\
&\le T_n + \int_{T_n}^\infty (K_n+p)\exp\Big\{ -{n\over 2K_n C_{2r}^2\sigma^2} t^2 \Big\}dt\\
& \le T_n + \sqrt{\pi\over 2}~ C_{2r}\sigma \sqrt{n^{-1}K_n}~ (K_n+p) n^{-{1\over 2r+1}} 
\\
& = O(T_n).
\end{align*}
In view of Proposition \ref{lem:bias} and the above result, we deduce that
\begin{align*}
\mathbb E\|\hat f_{(r)} - f\|_\infty &\le \|\bar f_{(r)} -   f \|_\infty + \mathbb E\|\hat f_{(r)} - \bar f_{(r)} \|_\infty=O\Big(L K_n^{-r} + \sigma \sqrt{\log n\over n} K_n^{1\over 2}\Big).
\end{align*}
This shows Statement (1) of Theorem \ref{thm:convergence} by using the optimal choice of $K_n$.



The next  proposition establishes uniform bounds for the stochastic estimation error for a fixed point as well as mean squared error.

\begin{proposition}\label{lem:pointwise}
For any given $x_0\in [0, 1]$, there exist two positive constants $C_{3r}$ and $C_{4r}$, which depend only on $r$, such that
\begin{eqnarray}
\mathbb E (|\hat f_{(r)}(x_0) - \bar f_{(r)}(x_0)|^2) & \le & C_{3r} \sigma^2 n^{-1}  K_n, \label{eqn:2nd_m} \\
\mathbb E( |\hat f_{(r)}(x_0) - \bar f_{(r)}(x_0)|^4) & \le & C_{4r} \sigma^4  n^{-2}  K_n^2. \label{eqn:4th_m}
\end{eqnarray}
Furthermore,
\begin{equation} \label{eqn:MSE}
    \mathbb E \big ( \big \|\hat f_{(r)} - \bar f_{(r)} \big \|^2_2 \big) :=  \mathbb E \big ( \int^1_0 \big | \hat f_{(r)}(x) - \bar f_{(r)}(x) \big |^2 dx \big) \, \le \, C_{3r} \sigma^2 n^{-1}  K_n.
\end{equation}
In particular, if $r\in[1, 2]$,  then $C_{3r}$ and $C_{4r}$ are independent of $r$.
\end{proposition}

\begin{proof}
 Recall that $p=\lceil r-1 \rceil$, $ N(x)=[B_1(x), \ldots, B_{K_n+p}(x)]^T \in \mathbb R^{K_n+p}$, and $X =[ N(x_1), \ldots, N(x_n)]^T \in \mathbb R^{ n\times (K_n+p)}.$ Fix $x_0\in [0, 1]$. Let $h := N(x_0)\in \mathbb R^{K_n+p}$.
Note that $h$ has
  at most $p$ nonzero elements and each of these nonzero elements is positive whose sum is $1$.
  Let $G_\alpha$ be the coefficient matrix of a linear selection function corresponding to an index set $\alpha$, i.e.,
$G_\alpha = F^T_\alpha (F_\alpha \Lambda F^T_\alpha)^{-1} F_\alpha$.
  Hence,
  \begin{eqnarray*}
     \| G_\alpha h \|^2_2 & = & \sum_i \Big(\sum_{h_j>0} (G_\alpha)_{ij} h_j \Big)^2 \, \le \,
      p \cdot \sum_i \sum_{h_j>0} (G_\alpha)^2_{ij} h^2_j \\
      & \le &  p \cdot \sum_i \sum_{h_j>0} (G_\alpha)^2_{ij} h_j  \, \le \, p \cdot \sum_{h_j>0} h_j  \sum_i (G_\alpha)^2_{ij} \\
      & \le &  p \cdot  \sum_{h_j>0} h_j \|G_\alpha\|^2_\infty \le p \cdot \|G_\alpha\|^2_\infty =p \cdot c^2_{\infty, p},
  \end{eqnarray*}
  where the first inequality in the third line is due to  the symmetry of $G_\alpha$ and the following implication
  \[
     \sum_i \big| (G_\alpha)_{ij} \big| \le \|G_\alpha\|_\infty \ \Longrightarrow \   \sum_i \big| (G_\alpha)_{ij} \big|^2 \le  \big( \sum_i \big| (G_\alpha)_{ij} \big| \big)^2 \le \| G_\alpha\|^2_\infty.
  \]
 As a result, in light of Theorem~\ref{theorem:uniform_Lip}, we have
 \begin{equation} \label{eqn:sup_bd_Ah}
  \max_\alpha \| G_\alpha h\|^2_2 \le p \cdot c^2_{\infty, p}.
 \end{equation}

Let $\epsilon =(\epsilon_1, \ldots, \epsilon_{n} )^T $ be iid random variables with mean zero and variance one, and let 
$z: =( f(x_1), \ldots, f(x_n))^T$. Thus $\bar y = X^T (z + \sigma \epsilon)/\beta_n$.  Hence
  \begin{eqnarray*}
    h^T \hat b( \bar y) =  h^T \hat b(X^T(z+ \sigma\epsilon)/\beta_n) = \frac{1}{\beta_n} h^T G_{\alpha(\epsilon)} X^T (z+ \sigma\epsilon).
%
  \end{eqnarray*}
 Furthermore, since $\hat b(\cdot)$ is a continuous piecewise linear function on $\mathbb R^{K_n+p}$, so is $\hat b\circ X^T$ on $\mathbb R^n$.  It follows from the polyhedral theory that $\hat b\circ X^T$ admits a conic subdivision of $\mathbb
R^n$ \cite{FPang03_book, Scholtes94_thesis}, i.e., there exist a finite collection of polyhedral cones  $\{ \mathcal C_j \}^q_{j=1}$ and linear functions $\{ g^j \}^q_{j=1}$ such that (i) $\bigcup_j \mathcal C_j = \mathbb R^n$; (ii) each cone $\mathcal C_j$ has nonempty interior; (iii) the intersection of any two cones is a common proper face of both cones; and (iv) $\hat b\circ X^T$ coincides with $ g^j$ on each $\mathcal C_j$. For any given $z' \in \mathbb R^n$, let $[z, z']$ be a line segment joining $z$ and $z'$. Starting from $z$, we assume that the line segment $[z, z']$ intersects some cones in the conic subdivision at $z_1, z_2, \ldots, z_{\ell-1} \in \mathbb R^n$, and ends at $z'$. Further, each subsegment of any two consecutive points, such as $[z, z_1], [z_1, z_2], \ldots, [z_{\ell-1}, z']$, belongs to a single cone. Hence there exist
$\mu_i \in [0, 1], i=1,\ldots, \ell$ with $\ell \le q$, $\sum^{\ell}_{i=1} \mu_i=1$ and $G_{\alpha_i}$ such that
 \begin{eqnarray*}
   \lefteqn{ \hat b(X^T z') - \hat b(X^Tz) }\\
    & =& G_{\alpha_1} X^T (z_1 -z) + G_{\alpha_2} X^T(z_2-z_1)+ \cdots +  G_{\alpha_{\ell}} X^T (z'-z_{\ell-1})\\
   & =&  \Big( \sum^\ell_{i=1} \mu_i G_{\alpha_i} \Big) X^T(z'-z),
  \end{eqnarray*}
where $\mu_i$ and $G_{\alpha_i}$ depend on $z'$ for the fixed $z$. Since there are $q$ cones in the conic subdivision, we may use the extended tuple $(\mu_i, G_{\alpha_i})^q_{i=1}$ (corresponding to $z'$) to characterize $\hat b(X^T z') -\hat b(X^T z) $, by setting some $\mu_i=0$, without loss of generality. Note that if $z'$ is a random variable, so is $(\mu_i, G_{\alpha_i})^q_{i=1}$.

Using $\bar y = X^T (z+\sigma \epsilon)/\beta_n$, we have, for the given vector $h$,
 \begin{eqnarray*}
   \lefteqn{ \mathbb E \big(|h^T \hat b (\bar y)  -h^T \hat b(\mathbb E(\bar y) ) |^2 \big) }   \\
    & =  & \frac{1}{\beta^2_n} \mathbb E_{ (\mu_i, G_{\alpha_i})^q_{i=1} }  \Big( \mathbb E \big (  | h^T \hat b (X^T(z+\sigma\epsilon)) -h^T \hat b(X^T z) |^2  \, \big | \,  (\mu_i, G_{\alpha_i})^q_{i=1}   \big )  \Big).
 \end{eqnarray*}
 Moreover,  for a fixed tuple $(\mu_i, G_{\alpha_i})^q_{i=1} $,
  \begin{eqnarray*}
   \lefteqn{  \mathbb E  \Big(  | h^T \hat b (X^T(z+\sigma\epsilon)) -h^T \hat b(X^T z) |^2  \, \big | \,  (\mu_i, G_{\alpha_i})^q_{i=1}   \Big) }  \\
   & = &  \sigma^2\mathbb E\Big( |\sum^q_{i=1} (\mu_i h^T G_{\alpha_i} )X^T \epsilon \cdot \sum^q_{i=1} ( \mu_i h^T G_{\alpha_i} )X^T \epsilon  |  \, \big | \,  (\mu_i, G_{\alpha_i})^q_{i=1}  \Big) \\
   &= & \sigma^2\mathbb E \Big(  |\sum^q_{i=1}  (\mu_i  h^T G_{\alpha_i}  X^T) \epsilon \cdot \epsilon^T (\sum^q_{j=1} \mu_j X G_{\alpha_j} h) |  \, \big | \,  (\mu_i, G_{\alpha_i})^q_{i=1}  \Big) \\
    & =  & \sigma^2 \Big | (\sum^q_{i=1}  \mu_i  h^T G_{\alpha_i} X^T  ) (\sum^q_{j=1} \mu_j X G_{\alpha_j} h)  \Big |  \\
    & \le  & \sigma^2 \Big(  \sum^q_{i=1} \mu_i \| X G_{\alpha_i} h\|_2 \Big) \cdot  \Big( \sum^q_{j=1 } \mu_j \| X G_{\alpha_j} h\|_2 \Big)  \\
   & \le & \sigma^2 \big( \max_i  \| X G_{\alpha_i} h\|_2\big)^2
       \, \le \, \sigma^2\|X\|^2_2 \cdot \big (\max_i \| G_{\alpha_i} h\|_2 \big)^2  \\
   & \le & \sigma^2\beta_n \lambda_{\max}(\Lambda) \cdot p \cdot c^2_{\infty, p}
 \end{eqnarray*}
 where the last inequality is due to $\| X\|^2_2 \le \beta_n \lambda_{\max}(\Lambda)$ and (\ref{eqn:sup_bd_Ah}). Therefore,
 \[
  \mathbb E \big (|h^T \hat b (\bar y) -h^T \hat b(\mathbb E(\bar y) |^2 \big)  \le \frac{1}{\beta_n} \cdot \lambda_{\max}(\Lambda) \cdot p \cdot c^2_{\infty, p} \cdot \sigma^2.
 \]
Observing that $\lambda_{\max}(\Lambda)$ is uniformly bounded \cite{zhou_98} and the uniform bound of $\beta_n$ in (\ref{eqn:beta_n}), we obtain (\ref{eqn:2nd_m}) for $p=\lceil r-1 \rceil$.
 %
%
%

The above argument can be extended to prove (\ref{eqn:MSE}). Indeed,
let $h(x):= N(x)$. Thus $\hat f_{(r)}(x) - \bar f_{(r)}(x) = h^T(x) [\hat b(\bar y) - \hat b(\mathbb E(\bar y))]$ such that
\begin{eqnarray*}
   \lefteqn{ \mathbb E \big( \| \hat f_{(r)} - \bar f_{(r)}  \|^2_2 \big) }   \\
    & =  & \frac{1}{\beta^2_n} \mathbb E_{ (\mu_i, G_{\alpha_i})^q_{i=1} }  \Big( \mathbb E \big ( \int^1_0  \big | h^T(x) [ \hat b (X^T(z+\sigma\epsilon)) - \hat b(X^T z)]  \big |^2  dx \, \big | \,  (\mu_i, G_{\alpha_i})^q_{i=1}   \big )  \Big),
 \end{eqnarray*}
where, for a given tuple $(\mu_i, G_{\alpha_i})^q_{i=1} $,
 \begin{eqnarray*}
 \lefteqn{  \mathbb E  \Big(  \int^1_0 \big | h^T(x)  [\hat b (X^T(z+\sigma\epsilon)) - \hat b(X^T z)] \big |^2  dx  \, \big | \,  (\mu_i, G_{\alpha_i})^q_{i=1}   \Big) }  \\
   & \le & \sigma^2 \| X\|^2_2  \int^1_0 \big(  \max_i  \| G_{\alpha_i} h(x)\|_2 \big)^2 dx  \, \le \, \sigma^2 \| X\|^2_2  \cdot (p \cdot c^2_{\infty, p} ),
\end{eqnarray*}
which yields (\ref{eqn:MSE}).

To show (\ref{eqn:4th_m}), we consider 
 \begin{eqnarray*}
   \lefteqn{ \mathbb E(| h^T \hat b (\bar y) -h^T \hat b(\mathbb E(\bar y) |^4) }   \\
    & =  & \frac{1}{\beta^4_n} \mathbb E_{ (\mu_i, G_{\alpha_i})^q_{i=1} }  \Big( \mathbb E \big (  | h^T \hat b (X^T(z+\sigma\epsilon)) -h^T \hat b(X^T z) |^4 \, \big | \,  (\mu_i, G_{\alpha_i})^q_{i=1}   \big )  \Big).
 \end{eqnarray*}
 For a fixed tuple $(\mu_i, G_{\alpha_i})^q_{i=1} $, let $v:=\sum^q_{i=1} \mu_i X G_{\alpha_i} h\in \mathbb R^n$ and $\mathbb E(\epsilon^4_i)=3, i=1, \ldots, n$.
  We thus have
  \begin{eqnarray*}
   \lefteqn{  \mathbb E  \Big(  | h^T \hat b (X^T(z+\sigma\epsilon)) -h^T \hat b(X^T z) |^4  \, \big | \,  (\mu_i, G_{\alpha_i})^q_{i=1}   \Big) }  \\
   & =  & \mathbb E\Big( | (\sum^q_{i=1} \mu_i h^T G_{\alpha_i} X^T) \sigma\epsilon \cdot ( \sum^q_{i=1}  \mu_i h^T G_{\alpha_i} X^T) \sigma\epsilon  |^2  \, \big | \,  (\mu_i, G_{\alpha_i})^q_{i=1}
   \Big)\\
   & =  &\mathbb E \Big(  | v^T \sigma\epsilon \cdot \sigma\epsilon^T v |^2  \, \big | \,  (\mu_i, G_{\alpha_i})^q_{i=1}  \Big)  \\
   & = & \sigma^4\sum^n_{i=1} v^4_i \cdot \mathbb E(\epsilon^4_i) +\sigma^4 \sum^n_{i, j=1, i\ne j} (v_i v_j)^2 \cdot  \mathbb E(\epsilon^2_i \cdot \epsilon^2_j)  \\
   &= & 2\sigma^4  \sum^n_{i=1} v^4_i + \sigma^4 \sum^n_{i=1} \sum^n_{j=1} v^2_i v^2_j   \\
  & \le  & 2\sigma^4  \|v\|^4_2 + \sigma^4 (\sum^n_{i=1} v^2_i) \cdot (\sum^n_{j=1} v^2_j)  \le 3 \sigma^4 \cdot \|v\|^4_2 \\
   & \le & 3\sigma^4 \cdot (\beta_n \lambda_{\max}(\Lambda) \cdot p \cdot c^2_{\infty, p})^2,
 \end{eqnarray*}
where the last inequality is due to $\| v\|^2_2 \le \beta_n \lambda_{\max}(\Lambda) p c^2_{\infty, p}$.  This shows that
\[
   \mathbb E \big( |h^T \hat b (\bar y) -h^T \hat b(\mathbb E(\bar y) |^4 \big) \le \frac{2}{\beta^2_n} \cdot   \sigma^4 \cdot ( \lambda_{\max}(\Lambda) \cdot p \cdot c^2_{\infty, p})^2.
\]
Using the uniform bounds on $\lambda_{\max}(\Lambda)$ and $\beta_n$ again, we obtain (\ref{eqn:4th_m}).
%
%
\end{proof}

 Propositions \ref{lem:bias} and \ref{lem:pointwise} imply that, for any $x_0\in [0,1]$,
\begin{equation}  \label{eqn:statement_2}
\sup_{f\in {\cal C}_H(r, L)} \mathbb E |\hat f_{(r)}(x_0) - f(x_0)|^2 \, \le \, C_{1r}^2 L^2 K_n^{-2r} + C_{3r}\sigma^2 K_nn^{-1} .
\end{equation}
This shows Statement (2) of Theorem \ref{thm:convergence} by using the optimal choice of $K_n$.

%
\section{Proof of Theorem \ref{thm:sup-norm}} \label{sect:proof_thm:sup-norm}

Throughout this section,
 we shall use $C_k$ or $c_k$ with $k\in \mathbb N$ to denote positive constants that depend only on $L$ (and $\sigma$). 
 We introduce some lemmas.
 The first lemma, as a complement to (\ref{equ:random2}),  provides a bound for the stochastic term of estimation error in the sup-norm.
%
%

\begin{lemma} \label{lem:random-2}
There exists a constant $C_2>0$ such that
\begin{align}
& \mathbb E\Big(\|\hat f_{(r)} - \bar f_{(r)}\|_\infty I\big\{\|\hat f_{(r)} - \bar f_{(r)}\|_\infty \ge u\big\} \Big) \label{equ:random3} \\
\le & \sqrt{\pi\over 2}C_2 \sigma {\sqrt{n^{-1}K_n}(K_n+1)} \exp\Big\{-{nu^2\over 2K_nC_2^2\sigma^2}   \Big\}. \nonumber
\end{align}
\end{lemma}

\begin{proof}
Direct calculation yields that
\begin{align*}
\lefteqn{ \mathbb E\Big(\|\hat f_{(r)} - \bar f_{(r)}\|_\infty I\big\{\|\hat f_{(r)} - \bar f_{(r)}\|_\infty \ge  u\big\} \Big) } \\
& =\int_u^\infty P\Big(\|\hat f_{(r)} - \bar f_{(r)}\|_\infty \ge t\Big) dt
 \le  \int_u^\infty (K_n+1)\exp\Big\{ -{n\over 2 K_n C_{2}^2\sigma^2}t^2 \Big\}dt\\
& \le \sqrt{\pi\over 2}C_2 \sigma {\sqrt{n^{-1}K_n}(K_n+1)} \exp\Big\{-{nu^2\over 2K_nC_2^2\sigma^2}   \Big\}.
\end{align*}
This completes the proof.
\end{proof}

It is shown next that it is highly improbable that the estimated $\hat r$ is strictly smaller than the true $r$. We say a few words about notation.
Recall  that $\tau_n = \lceil (\log n)^{1 \over 2} \rceil$ and define the set $\mathcal R:=\{ r_j \, | \, r_j =1 + j/\tau_n, \, j=0, 1,\ldots, \tau_n \}$.

\begin{lemma}\label{lem:est-r} Let $r, d \in [1, 2]$ with $d<r$. There exists $C_3>0$ such that
$$\sup_{f\in {\cal C}_H(r,L)} P\big(\hat r = d\big) \le C_3 \tau_n n^{-{1\over 2d+1}}.$$
\end{lemma}

\begin{proof}
By the definition of $\hat r$ given before Theorem \ref{thm:sup-norm},
$$\sup_{f\in {\cal C}_H(r,L)}P\big( \hat r  =  d\big) \le \sum_{d \, \ge \, r'\in \mathcal R} \sup_{f\in {\cal C}_H(r,L)} p_\infty(r', d) \le \tau_n~ \max_{d\, \ge \, r'\in \mathcal R} \sup_{f\in {\cal C}_H(r,L)} p_\infty(r', d),$$
where
$$p_\infty(r', d) := P\Big( \|\hat f_{(r'')} -\hat f_{(r')}\|_\infty > {1+\sqrt{2} \over 2}~\psi_{(r')} \Big).$$
Here $\psi_{(r')}$ is defined in (\ref{eqn:psi}) and
  $r'' := \min\{r\in \mathcal R \, | \, r>d\}$, i.e., $r''\in \mathcal R$ is closest to $d$ from above. Hence, $r''>d\ge r'$. In view of (\ref{equ:bias2}) and (\ref{equ:random2}),
\begin{align*}
\|\hat f_{(r'')} -\hat f_{(r')}\|_\infty & \le \|\hat f_{(r'')} - \bar f_{(r'')}\|_\infty + \|\hat f_{(r')} - \bar f_{(r')}\|_\infty + \|\bar f_{(r'')} - f\|_\infty + \|\bar f_{(r')} - f\|_\infty\\
&\le  \|\hat f_{(r')} - \bar f_{(r')}\|_\infty + \|\hat f_{(r'')} - \bar f_{(r'')}\|_\infty +C_{1} L K_{(r')}^{-r'} + C_{1} L K_{(r'')}^{-r''}\\
& =  C_{1} L  K_{(r')}^{-r'} \big(1 +\omega_{r',r''}\big)+  \|\hat f_{(r')} - \bar f_{(r')}\|_\infty + \|\hat f_{(r'')} - \bar f_{(r'')}\|_\infty,
\end{align*}
where
\begin{align*}
|\omega_{r',r''}| \, = \, {K^{r'}_{(r')}\over K^{r''}_{(r'')}} \le c_1 \Big({\log n\over n}\Big)^{(r''-r') \over (2r''+1)(2r'+1) }   \le c_1 \Big({\log n\over n}\Big)^{1\over 25\tau_n },
\end{align*}
for a positive constant $c_1$ which is bounded away from zero and above.  Since $(n^{-1} \log n)^{\sqrt{\log n}} \rightarrow 0$ as $n \rightarrow \infty$,  $\omega_{r',r''}$ converges to zero uniformly for all $r',r'' \in \mathcal R$ with the given $\tau_n$.
Let $p_\infty(r', d) \le p_{1,\infty} + p_{2,\infty},$
where
\begin{align*}
p_{1,\infty} &:= P\Big\{  \|\hat f_{(r')} - \bar f_{(r')}\|_\infty  \ge  {\sqrt{2}\over 2}~\psi_{(r')}( 1- o(1))   \Big\}, \\
p_{2,\infty} &:= P\Big\{ \|\hat f_{(r'')} - \bar f_{(r'')}\|_\infty  \ge {(2r''+1) \over 2} ~\psi_{(r'')} \Big\}.
\end{align*}
By using (\ref{equ:random2}), (\ref{eqn:psi}) and the orders of $K_{(r')}, K_{(r'')}$, we obtain two positive constants $c_2$ and $c_3$ such that
\begin{align*}
p_{1,\infty} & \le \big ( K_{(r')} +1 \big) \cdot n^{- {2 \over 2 r'+1}} \le c_2 n^{-{1\over (2r'+1)}} (\log n)^{-{1\over 2r'+1}} \le c_2 n^{-{1\over 2d+1}} (\log n)^{-1/5}, \\
%
%
 p_{2,\infty} &\le \big ( K_{(r'')} +1 \big) \cdot n^{- (2 r''+1)} \le c_3 n^{-{1\over (2r''+1)}} (\log n)^{-{1\over 2r''+1}} \le c_3 n^{-{1\over 2d+1}} (\log n)^{-1/5}.
 %
\end{align*}
%
%
Combining the above results, we see that the lemma holds.
\end{proof}

The following lemma develops a uniform bound on the sup-norm risk of $\hat f_{(r)}$ for $r\in [1,2]$.

\begin{lemma} \label{lem:error-bd}
There exist positive constants $C_4$ and $C_{5}$ such that
\begin{eqnarray*}
\sup_{r\in[1,2]} \, \sup_{f\in {\cal C}_H(r, L)}P_f\Big\{  \psi_{(r)}^{-1}\|\hat f_{(r)} - f\|_\infty \ge  1+\sqrt{2}~  \Big\} & \le &C_4 n^{-1},\\
\sup_{r\in [1,2]} \, \sup_{f\in {\cal C}_H(r, L)} \mathbb E\Big( \psi_{(r)}^{-2}\|\hat f_{(r)} - f\|_\infty^2 \Big) &\le & C_{5}. 
\end{eqnarray*}
\end{lemma}

\begin{proof}
Since
\begin{align*}
\|\hat f_{(r)} - f\|_\infty &\le \|\hat f_{(r)} - \bar f_{(r)}\|_\infty + \|\bar f_{(r)} - f\|_\infty\le  \|\hat f_{(r)} - \bar f_{(r)}\|_\infty + C_{1} L K_{(r)}^{-r},
\end{align*}
we have via (\ref{eqn:psi}) and $C_1 L K^{-r}_{(r)} = \psi_{(r)}/2$ that, for any $u\ge 1+\sqrt{2}~$,
\begin{align*}
\sup_{f\in {\cal C}_H(r, L)}P\Big\{  \psi_{(r)}^{-1}\|\hat f_{(r)} - f\|_\infty \ge  u \Big\} &\le P\Big(\|\hat f_{(r)} - \bar f_{(r)}\|_\infty \ge \big(u-{1\over 2} \big)\psi_{(r)}\Big)\\
&\le (K_{(r)}+1) e^{-\log n \cdot {(2u-1)^2 \over 2r+1}}.
\end{align*}
In view of (\ref{eqn:K_r}),  $ K_{(r)}=  q(r) (n/\log n)^{1/(2r+1)} $ for some function $q(\cdot)$ that is positive and continuous on $[1, 2]$. Hence,  a constant $q_*>0$ exists such that $K_{(r)} \le q_*  (n/\log n)^{1/(2r+1)}$ for all $r\in[1,2]$.
Applying this to $u=1+\sqrt{2}$, we obtain a positive constant $C_4$
%
 %
 such that for any $r\in[1,2]$,
\begin{align*}
\sup_{f\in {\cal C}_H(r, L)}P\Big\{  \psi_{(r)}^{-1}\|\hat f_{(r)} - f\|_\infty \ge 1+\sqrt{2}  \Big\} &\le C_4 n^{-{4\sqrt{2}+8\over 2r+1}}(\log n)^{-{1\over 2r+1}}\le C_4 n^{-1}.
\end{align*}
Using the above results,  we further get a positive constant $c_5$ such that 
\begin{align*}
\mathbb E\Big(\psi_{(r)}^{-2}\|\hat f_{(r)} - f\|_\infty^2 \Big) &\le (1+\sqrt{2}~) + \int_{(1+\sqrt{2})}^\infty P\Big( \psi_{(r)}^{-2}\|\hat f_{(r)} - f\|_\infty^2 \ge t  \Big)dt\\
&\le (1+\sqrt{2}) + \int_{(1+\sqrt{2})}^\infty (K_{(r)}+1) \exp\Big\{- {(2t^{1/2}-1)^2 \over 2r+1} \log n\Big\}dt\\
& \le (1+\sqrt{2}) + c_5 \cdot n^{- {1 \over 2r +1 }} \le (1+\sqrt{2}) + c_5 =: C_5.
\end{align*}
 This completes the proof.
\end{proof}

We are ready to prove the theorem.
Observe that
$$\sup_{f\in {\cal C}_H(r, L)} \mathbb E\Big\{\psi_{(r)}^{-1}\|\hat f_{(\hat r)} - f\|_\infty\Big\} \le R^{-} + R^{+}, $$
where
$
R^{-} \,  := \,  \sup_{f\in {\cal C}_H(r, L)} \mathbb E\Big\{\psi_{(r)}^{-1}\|\hat f_{(\hat r)} - f\|_\infty ~ I\{\hat r < r\}\Big\}$,
and
$
R^+ \,:= \, \sup_{f\in {\cal C}_H(r, L)} \mathbb E\Big\{\psi_{(r)}^{-1}\|\hat f_{(\hat r)} - f\|_\infty~ I\{\hat r \ge r\}\Big\}.
$
%
%
Hence, it suffices to show that
\begin{eqnarray}
\limsup_{n\rightarrow\infty} \sup_{r\in [1, 2]} R^{-}  & = & 0, \label{equ:low} \\
\limsup_{n\rightarrow\infty} \sup_{r\in [1, 2]} R^{+} & < & \infty. \label{equ:high}
\end{eqnarray}


We first prove (\ref{equ:low}). By the definition of $\hat r$ given before Theorem  \ref{thm:sup-norm},
\[
R^- \le \sum_{r \, > \, r'\in \mathcal R} \sup_{f\in {\cal C}_H(r, L)} \mathbb E \Big( \psi_{(r)}^{-1}\|\hat f_{(\hat r)} - f\|_\infty ~ I\{\hat r = r'\}\Big) \le \rho_1 + \rho_2,
\]
where, in view of  (\ref{equ:bias2}) and $C_1 L K^{-r'}_{(r')} = \psi_{(r')}/2$ (for $\rho_2$ below),
\begin{align*}
\rho_1 &:= \sum_{r \, > \, r'\in \mathcal R} \sup_{f\in {\cal C}_H(r, L)} P\big( \hat r = r'  \big) \psi_{(r)}^{-1} \Big( {1\over 2}\psi_{(r')} + {\sqrt{2}\over 2}\psi_{(r')}    \Big),\\
\rho_2 &:= \sum_{r \, > \, r'\in \mathcal R}  \psi_{(r)}^{-1} \mathbb E\Big( \Big[  {\psi_{(r')} \over 2} +  \|\hat f_{(r')} -\bar f_{(r')}\|_\infty  \Big] I\{\|\hat f_{(r')} -\bar f_{(r')}\|_\infty \ge {\sqrt{2}\over 2}\psi_{(r')}\}  \Big).
\end{align*}
%
We will prove $\sup_{r\in [1,2]}\rho_j = o(1)$ as $n\rightarrow \infty$, $j=1,2$.
%
%
Since $r'<r$, it is easy to see that there exists a positive constant $c_6$ such that
\[
 {\psi_{(r')}\over \psi_{(r)}} \le c_6 \Big(  {\log n\over n} \Big)^{(r-r')\over (2r+1)(2r'+1)}\le c_6
  \Big(  {\log n\over n} \Big)^{r -r' \over 5(2r+1)} \le c_6.
\]
It thus follows from Lemma \ref{lem:est-r} that, as $n\rightarrow\infty$,
\[
 \rho_1 \le  \tau_n \cdot C_3 \tau_n  n^{-{1\over 2r+1}}  \cdot  {(1+\sqrt{2}) c_6 \over 2}
  \le C_3 {(1+\sqrt{2}) c_6 \over 2}  \log n \cdot n^{-{1\over 2r+1}} \longrightarrow 0.
%
\]
Further, from (\ref{equ:random2}), we deduce the existence of a constant $c_7>0$ such that
\begin{eqnarray*}
\lefteqn{ \sum_{r \, > \, r'\in \mathcal R}  \psi_{(r)}^{-1} \mathbb E\Big(\psi_{(r')}I\{\|\hat f_{(r')} -\bar f_{(r')}\|_\infty \ge {\sqrt{2}\over 2}\psi_{(r')}\} \Big) } \\
&=&\sum_{r \, > \, r'\in \mathcal R}  {\psi_{(r')}\over \psi_{(r)}} P\Big\{\|\hat f_{(r')} -\bar f_{(r')}\|_\infty \ge {\sqrt{2}\over 2}\psi_{(r')}\Big\} 
%
 \, \le \, \tau_n \cdot c_6 \cdot c_7 \cdot n^{-1 \over 2r' +1}.
\end{eqnarray*}
Also, it follows from Lemma \ref{lem:random-2} that a constant $c_8>0$ exists such that for all large $n$,
\begin{align*}
\lefteqn{ \sum_{r \, > \, r'\in \mathcal R}  \psi_{(r)}^{-1}  \mathbb E\Big(   \|\hat f_{(r')} -\bar f_{(r')}\|_\infty I\{\|\hat f_{(r')} -\bar f_{(r')}\|_\infty \ge {\sqrt{2}\over 2}\psi_{(r')}\}  \Big) }  \qquad \qquad  \qquad \\
&= \sum_{r \, > \, r'\in \mathcal R}  \psi_{(r)}^{-1}  c_8 \big( \log n \big)^{{1 \over 2(2 r'+1)}}
\cdot n^{-{r'+3 \over 2 r'+1}}
 \, \le \,    \tau_n \psi_{(r)}^{-1}  c_8 n^{-{1 \over 2}}. 
\end{align*}
%
%
%
By virtue of the above results, we have $\rho_2 \rightarrow 0 $ as $n\rightarrow \infty$. This yields (\ref{equ:low}).
%
%

We now prove (\ref{equ:high}). Consider the random event
$\aleph(r, r') := \big\{\psi_{(r)}^{-1} \|\hat f_{(r')} - f\|_\infty\ge 1+\sqrt{2}, \ r' \in \mathcal R \, \big\}.$
Then
\begin{align*}
R^+ &\le \sup_{f\in {\cal C}_H(r, L)}\sum_{r \le r' \in \cal R} \mathbb E \Big(\psi_{(r)}^{-1}\|\hat f_{(r')} - f\|_\infty I\{\hat r = r' \}\Big)\\
&\le  (1+\sqrt{2}\,) \sup_{f\in {\mathcal C}_H(r, L) } P \{\hat r\ge r\}\\
&~~~+ \sum_{r\le r' \in \cal R }\sup_{f\in {\cal C}_H(r, L)} \mathbb E \Big(\psi_{(r)}^{-1}\|\hat f_{(r')} - f\|_\infty I\{\{\hat r = r'\}\cap \aleph(r, r')\}\Big)\\
&\le (1+\sqrt{2}\,)  + \sum_{r\le r' \in \cal R}\left(\sup_{f\in {\cal C}_H(r, L)} ( \mathbb  E \big(\psi_{(r)}^{-2}\|\hat f_{(r')} - f\|_\infty^2 \big))^{1/2} \sup_{f\in {\cal C}_H(r, L)} \rho_f^{1/2}(r, r')\right),
\end{align*}
where $\rho_f(r, r') := P(\{\hat r= r'\} \cap \aleph(r, r'))$ for $r'\in \cal R$.
Let $r_*:=\min\{ r'\in \mathcal R \, : \, r'>r\}$. Hence, $r_*-r\in [0, 1/\tau_n)$.  In view of (\ref{eqn:psi}), we have $\psi_{(s}) =p(s) \cdot (n^{-1} \log n)^{s/(2s +1)}$ for a function $p(\cdot)$ that is positive and continuous on $[1,2]$. Let $p_*:=\min_{s\in[1,2]} p(s)>0$. This shows that
\[
 \psi_{(r_*)} = p(r_*)  \Big({\log n \over n} \Big)^{{r_* \over (2r_* +1)} } \le \Big[ 1 + {p(r_*)- p(r) \over p_*} \Big] p(r)   \Big({\log n \over n} \Big)^{{r \over (2r +1)} } = \big[1 + o(1) \big] \psi_{(r)}.
\]
Hence, if $\hat r =r' \ge r$, then, by the definition of $\hat r \in \mathcal R$, $\|\hat f_{(r')} - \hat f_{(r_*)}\|_\infty \le {1\over 2}(1+\sqrt{2})\psi_{r_*} \le {1\over 2}(1+\sqrt{2} +o(1)) \psi_{(r)}$. Therefore,
\begin{align*}
\psi_{(r)}^{-1}\|\hat f_{(r')} - f\|_\infty & \le  \psi_{(r)}^{-1} \cdot \Big( \| \hat f_{(r')} - \hat f_{(r_*)} \|_\infty + \|\hat f_{(r_*)} - f\|_\infty \Big) \\
& \le {1\over 2}\big(1+\sqrt{2}+o(1) \big)  + \psi_{(r)}^{-1} \psi_{(r_*)} \cdot \big[  \psi^{-1}_{(r_*)}  \|\hat f_{(r_*)} - f\|_\infty \big] \\
& \le {1\over 2}\big (1+\sqrt{2}+o(1) \big) + [1+o(1)] \cdot \psi_{(r_*)}^{-1}\|\hat f_{(r_*)} - f \|_\infty.
\end{align*}
It follows from Lemma \ref{lem:error-bd} that for all large $n$,
\[
\mathbb E\Big\{\psi_{(r)}^{-2}\|\hat f_{(r')} - f\|_\infty^2\Big\}\le 2\big[4 + (1+o(1))\cdot \mathbb E  \big(\psi_{(r_*)}^{-2}\|\hat f_{(r_*)} - f  \|_\infty^2) \big] \le 2(4+2C_5).
\]

We consider $\rho_f(r, r')$ next. By Lemma \ref{lem:error-bd}, we have, when $r'=r$,
\[
\sup_{f\in {\cal C}_H(r, L)} \rho_f(r, r)\le \sup_{f\in {\cal C}_H(r, L)} P\{ \aleph(r, r) \}\le C_4n^{-1}.
\]
Now consider $\rho_f(r, r')$ with $r'>r$.  Since $\bar f_{(s)}(\cdot)$ is a continuous function  for any $s$, there exists a nonrandom point $t_*\in [0, 1]$ such that
$|\bar f_{(r')}(t_*) - \bar f_{(r_*)}(t_*)| = \|\bar f_{(r')} - \bar f_{(r_*)}\|_\infty.$
Let  $\xi_* := (\hat f_{(r_*)}(t_*) - \bar f_{(r_*)}(t_*)) - (\hat f_{(r')}(t_*) - \bar f_{(r')}(t_*))$
Therefore, if $\hat r=r'>r$, then
\begin{align*}
  \|\bar f_{(r')} - \bar f_{(r_*)}\|_\infty &\le \big |\hat f_{(r')}(t_*) - \hat f_{(r_*)}(t_*) \big| + | \xi_*| \le \|\hat f_{(r')} - \hat f_{(r_*)}\|_\infty  + | \xi_*| \\
  & \le  {1\over 2}\big(1+\sqrt{2}+o(1) \big) \psi_{(r)} + |\xi_*|.
\end{align*}
Using this result and  $\|\bar f_{(r_*)} - f\|_\infty\le \psi_{(r_*)}/2$, we have
\begin{align*}
\|\hat f_{(r')} - f\|_\infty & \le \|\bar f_{(r')} - \bar f_{(r_*)}\|_\infty +  \|\hat f_{(r')} - \bar f_{(r')}\|_\infty+ \|\bar f_{(r_*)} - f\|_\infty \\
& \le {1\over 2}\big(1+\sqrt{2}+o(1) \big) \psi_{(r)} + |\xi_*| + \|\hat f_{(r')} - \bar f_{(r')}\|_\infty +{1+o(1) \over 2}\psi_{(r)}.
\end{align*}
%
%
As a result, for $r<r' \in \mathcal R$, we further deduce via Markov Inequality,
\begin{align*}
\rho_f(r, r') & \le P\Big( |\xi^*| + \|\hat f_{(r')} - \bar f_{(r')}\|_\infty \ge {\sqrt{2}+o(1) \over 2}\psi_{(r)} \Big)\\
& \le P\Big( \|\hat f_{(r')} - \bar f_{(r')}\|_\infty \ge {1.1 \over 2}\psi_{(r_*)} \Big) +  P\Big( |\xi_*|\ge {\sqrt{2}-1.1+o(1) \over 2}\psi_{(r_*)} \Big) \\
& \le P\Big( \|\hat f_{(r')} - \bar f_{(r')}\|_\infty \ge {1.1 \over 2}\psi_{(r_*)} \Big) +  10^2 \cdot \psi_{(r_*)}^{-2}  \cdot\mathbb E( |\xi_*|^2).
%
\end{align*}
It follows from (\ref{equ:random2}) and $ r_* \le r' \le 2 $ that there exist constants $c_9>0$ and $\alpha>0$ (independent of $r\in [1, 2]$) such that  $ r'> r_* \Rightarrow K_{(r_*)} / K_{(r')} \ge   \alpha (n/\log n)^{ 2/(25 \tau_n)} \ge 1$ for all large $n$ and  that 
\begin{align*}
  P\Big( \|\hat f_{(r')} - \bar f_{(r')}\|_\infty \ge {1.1 \over 2}\psi_{(r_*)} \Big) & \le c_9 \Big({n\over \log n} \Big)^{1\over (2r'+1)} \cdot  \exp\Big\{ - {1.21 \over 2 r_*+1}  \cdot { K_{(r_*)} \over K_{(r')} } \log n \Big\}  \\
  & \le c_9 n^{ {1\over (2r'+1)} - {1.21 \over (2r_*+1)} } \le c_9 n^{-{0.21 \over (2r_*+1)} } \le c_9 n^{-{0.2 \over 5} }.
\end{align*}
%
Using Proposition \ref{lem:pointwise}, we also obtain a constant $c_{10}>0$ such that
%
\begin{align*}
\mathbb E( |\xi_*|^2) & \le 2 \Big[ \mathbb E( |\hat f_{(r_*)} (t_*) -  \bar f_{(r_*)} (t_*) |^2) +  \mathbb E( |\hat f_{(r')} (t_*) -  \bar f_{(r')} (t_*) |^2)  \Big] \\
& \le  c_{10} { K_{(r_*)} + K_{(r')} \over n} \le 2 \cdot c_{10} {K_{r(_*)} \over n}.
\end{align*}
Noting that $\psi^2_{(r_*)} \ge  \wt\alpha (K_{(r_*)}/n) \log n$ for a positive constant $\wt\alpha$ independent of $r_*$, we have $\psi_{(r_*)}^{-2}  \cdot  \mathbb E( |\xi_*|^2) \le 2 c_{10} \big (\wt\alpha  \cdot  \log n \big)^{-1} $. Hence,  there exists a constant $\wt c_{10}>0$ (independent of $r, r'\in[1,2]$) such that for all $n$ sufficiently large,
\[
 \rho^{1/2}_f (r, r') \le \wt c_{10} \cdot \big(\log n \big)^{-1/2}   \ \Longrightarrow \  R^+ \le (1+\sqrt{2}) + \sqrt{2( 4+ 2 C_5)} \cdot \wt c_{10}.
\]
This leads to  (\ref{equ:high}), and thus completes the proof of Theorem (\ref{thm:sup-norm}).
%


\section{Proof of Theorem \ref{thm:pointwise}}  \label{sect:proof_pointewise}

We establish  the following lemma to be used for the analysis of the risk of $\wt f(x_0)$.

\begin{lemma}\label{lem:fj1}
Suppose that $f$ is convex and differentiable on $[0, 1]$. Then there exists a positive constant $C_{6}$ independent of $f$ such that for each $j$,
\begin{equation}\label{equ:fj}
\mathbb E \big( |\wt f_j(x_0) - f(x_0)|^2 I_j \big) \, \le \, C_{6} 2^{j} n^{-{4 \over 5} }\sigma^2.
\end{equation}
\end{lemma}

\begin{proof} 
%
Recall that $\zeta_k = {1\over n}((k-1)M_n+{M_n+1 \over 2}), \forall \, k=1, \ldots, K_{n, j}$
%
%
Let $\grave f$ and $\check f$ be two piecewise linear functions such that $\grave f(\zeta_k) = \mathbb E(\bar y_{k, j}) $ and $\check f(\zeta_k) = f(\zeta_k)$, respectively. Note that if $M_n$ is odd, then  $\zeta_k$ is a design point so that
$\mathbb E(\bar y_k) - f(\zeta_k)=0$. Otherwise, direct calculation yields that
\begin{align*}
&\mathbb E(\bar y_k) - f(\zeta_k)  = {1\over M_n} \sum_{j=1}^{M_n} \Big[f(x_{(k-1)M_n+j}) -f(\zeta_k)\Big]\\
=&{1\over M_n} \sum_{j=1}^{M_n/2} \Big[\big(f(x_{(k-1)M_n+M_n-j+1}) - f(\zeta_k)\big) - \big(f(\zeta_k)-f(x_{(k-1)M_n+j})\big)\Big]\\
=&{1\over M_n} \sum_{j=1}^{M_n/2} {{M_n\over 2} - j\over n} \Bigg [{f(x_{(k-1)M_n+M_n-j+1}) - f(\zeta_k)\over {({M_n\over 2} - j)/ n}} - {f(\zeta_k)-f(x_{(k-1)M_n+j})\over {({M_n\over 2} - j)/ n}}\Bigg].
\end{align*}
Since $f$ is a convex function,  we have $\mathbb E(\bar y_k) - f(\zeta_k) \ge 0$ and
\begin{align*}
\mathbb E(\bar y_k) - f(\zeta_k) & \le {1\over M_n} \sum_{j=1}^{M_n/2} {M_n/2\over n} \Bigg [{f(\zeta_{k+1}) - f(\zeta_k)\over 1/K_{n,j}} - {f(\zeta_k)-f(\zeta_{k-1})\over 1/K_{n,j}}\Bigg ]\\
& = {1\over 4} \Delta^2 f(\zeta_{k+1}).
\end{align*}
Hence, for any $x_0\in (\zeta_{d_n}, \zeta_{d_n+1}]$,
\begin{align*}
0\le \grave f(x_0) - \check f(x_0) &\le {1\over 4}\max\Big\{ \Delta^2 f(\zeta_{d_n+1}), \Delta^2 f(\zeta_{d_n+2}) \Big\} \le \Delta^2 f(\zeta_{d_n+1}) + \Delta^2 f(\zeta_{d_n+2}) \\
& \le \Delta f(\zeta_{d_n+2}) - \Delta f(\zeta_{d_n}).
\end{align*}
Further,  since $f$ is convex, $f(x_0) - f(\zeta_{d_n}) \ge f'(\zeta_{d_n})(x_0-\zeta_{d_n})$ such that
\begin{align*}
0\le \check f(x_0) - f(x_0) &= f(\zeta_{d_n})+K_n(f(\zeta_{d_n+1}) - f(\zeta_{d_n}))(x_0 - \zeta_{d_n}) -f(x_0) \nonumber\\
& \le [ f(\zeta_{d_n+1}) - f(\zeta_{d_n})-K_n^{-1}f'(\zeta_{d_n})] K_n(x_0-\zeta_{d_n}) \nonumber\\
& \le [ f(\zeta_{d_n+1}) - f(\zeta_{d_n})] - [f(\zeta_{d_n}) - f(\zeta_{d_n-1})] \\
& \le \Delta f(\zeta_{d_n+2}) - \Delta f(\zeta_{d_n}).
\end{align*}
Let $\tau_j := \mathbb E(\Delta\bar y_{d_n+4, j} - \Delta \bar y_{d_n-2, j})$.  Hence,
$0\le \grave f(x_0) - f(x_0) \le 2 \tau_j$.

Notice that for $x_0\in (\zeta_{d_n}, \zeta_{d_n+1}]$, there exists $\mu \in(0, 1]$ such that $\wt f_j(x_0)= \mu \hat f (\zeta_{d_n}) + (1-\mu) \hat f (\zeta_{d_n+1})$ and $\grave f (x_0)= \mu \bar f (\zeta_{d_n}) + (1-\mu) \bar f (\zeta_{d_n+1})$, where $\hat f$ and $\bar f$ are the piecewise constant  splines (i.e., $p=0$) corresponding to the convex constrained least squares  for $(\bar y_{k, j})$ and $(\mathbb E(\bar y_{k, j}))$, respectively. It follows  from Proposition \ref{lem:pointwise} that  a positive constant $c_{11}$ exists such that
\begin{align}
  & \mathbb E[ (\wt f_j(x_0) - \grave f(x_0) )^2 I_j] \, \le \, 2 \big(  \mu^2 \mathbb E | \hat f(\zeta_{d_n}) - \bar f(\zeta_{d_n})|^2 + (1-\mu)^2  \mathbb E | \hat f(\zeta_{d_n+1}) - \bar f(\zeta_{d_n+1})|^2 \big)  \label{eqn:wt_f-grave f}  \\
  & \qquad \, \le \, 2 [ \mu + (1-\mu)] C_{3} \sigma^2 K_{n,j} \cdot n^{-1} \, \le \, c_{11} \sigma^2 2^j n^{-4/5}. \nonumber
\end{align}
%
This implies
\begin{align*}
& \mathbb E[(\wt f_j(x_0) - f(x_0))^2 I_j] \, \le  \, 2 \Big(  \mathbb E[ (\grave f_j(x_0) - f(x_0) )^2 I_j] +  \mathbb E[ (\wt f_j(x_0) - \grave f(x_0) )^2 I_j] \Big) \\
 & \qquad   \le \, 8 \tau_j^2  \mathbb E (I_j) + 2  \mathbb E[(\wt f_j(x_0) - \grave f(x_0))^2]  \le  8\tau_j^2  \mathbb E (I_j)  + c_{11}\sigma^2 2^{j+1} n^{-4/5}.
\end{align*}
%
If $\tau_j \le 2\lambda 2^{j/2+1} n^{-2/5}\sigma$, then (\ref{equ:fj}) holds.
We thus only consider the case when $\tau_j > 2\lambda 2^{j/2+1} n^{-2/5}\sigma$. In this case, note that $\Delta\bar y_{d_n+4, j} - \Delta \bar y_{d_n-2, j}$  has a normal distribution with mean $\tau_j$ and variance $2^{j+2}n^{-4/5}\sigma^2$. Consequently,
\begin{align*}
 \mathbb E (I_j)  &\le \mathbb E I\Big(\Delta\bar y_{d_n+4, j} - \Delta \bar y_{d_n-2, j} \le \lambda 2^{ {j \over 2}+1} n^{-{2 \over 5} } \sigma \Big)\\
& \le P\Bigg(Z \le \lambda - {\tau_j n^{{2 \over 5} }\over 2^{{ j \over 2}+1}\sigma}\Bigg)\le P\Bigg(Z \le - {\tau_j n^{{2 \over 5} }\over 2^{{ j \over 2}+2}\sigma}\Bigg)\le \exp\Bigg(-{\tau_j^2 n^{{4 \over 5}}\over 2^{(j+5)}\sigma^2}\Bigg),
\end{align*}
where $Z$ is a standard normal random variable.
In view of $\sup_{z>0} z \exp(-{\rho^2 z \over 2}) = 2 \rho^{-2}e^{-1}$, we obtain $\tau_j^2 \mathbb E (I_j) \le 2e^{-1} 2^{j+4}n^{-4/5}\sigma^2$. Hence (\ref{equ:fj}) holds.
\end{proof}


With this lemma, we show as follows:

\begin{proof}[Proof of Theorem \ref{thm:pointwise}]
%
%
%
Recall that $K_{n, j} := 2^j n^{1/5}$.
If $f\in {\cal C}_H(r, L)$ with $r\in [1, 2]$, then
$\tau_j = \Delta f(\zeta_{{d_n}+2}) - \Delta f(\zeta_{{d_n}}) \le 2 L K_{n, j}^{-r} = 2 L (2^{-j}n^{-1/5})^{r}.$
Let $J$ be the smallest natural number (dependent on $r$) such that
\begin{equation}\label{equ:J}
2^J n^{-{4 \over 5} } \ge L^{{ 2 \over 2r+1} } \sigma^{-{ 2 \over 2r+1} } n^{-{2r \over 2r+1} }.
\end{equation}
If $j\ge J$, then $2^{j r} \ge L^{2r/(2r+1)} \sigma^{-2r/(2r+1)} n^{(4r-2r^2)/(10r+5)}$.
%
%
Hence, this shows via (\ref{equ:J}) that, for $j\ge J$,
\begin{equation}\label{equ:tauJ}
     \tau_j \le 2 L \cdot 2^{-jr} \cdot n^{-{r \over 5} }  \le 2  L^{ {1 \over 2r+1}}\sigma^{{2r \over 2r+1} } n^{- {r \over 2r+1}}\le 2 \cdot 2^{{J \over 2} }n^{-{2 \over 5} }\sigma.
\end{equation}

Recall that $P(Z > \lambda)< 1/4$ for the positive $\lambda$. Then there exists a $\delta>0$ such that $\gamma:=P(Z\ge \lambda -\delta)<1/4$. Note that $4\gamma<1$. Given this $\delta$, choose $K \in \mathbb N$ (independent of $r$) such that $1<\delta \cdot 2^{K/2}$.  Since only one $I_j\ne 0$ and $I^2_j = I_j$, $|\wt f(x_0) - f(x_0)|^2 = \sum^\infty_{j=0} (\wt f_j (x_0) - f(x_0))^2 I_j  $ such that the risk of $\wt f$ is decomposed into two terms:
\begin{equation}\label{equ:sum}
\mathbb E|\wt f(x_0) - f(x_0)|^2 = \sum_{j=0}^{J+K} \mathbb E [(\wt f_j(x_0) - f(x_0))^2 I_j] + \sum_{j=J+K+1}^\infty \mathbb E[ (\wt f_j (x_0) - f(x_0))^2  I_j ].
\end{equation}

We consider the first sum in (\ref{equ:sum}).
It follows from Lemma \ref{lem:fj1} that
\[
\sum_{j=0}^{J+K} \mathbb E [ (\wt f(x_0) - f(x_0))^2 I_j ]   \, \le \, \sum_{j=0}^{J+K} C_6 2^j n^{-{4 \over 5} }  \sigma^2 \, \le \, C_6 2^{K+1} 2^{J}n^{-{4 \over 5} }\sigma^2.
\]
Since $J$ is the smallest integer satisfying (\ref{equ:J}), we have
\begin{equation}\label{equ:J2}
2^J n^{-{4 \over 5} } \, \le \, 2 \cdot L^{ {2\over 2r+1 } }\sigma^{-{ 2\over 2r+1} }  n^{-{2r \over 2r+1} }.
\end{equation}
Therefore,
\begin{equation}\label{equ:sum1}
\sum_{j=0}^{J+K} \mathbb E [(\wt f(x_0) - f(x_0))^2 I_j] \,  \le \, C_6 2^{K+2}L^{ {2 \over 2r+1} }\sigma^{{ 4r \over 2r+1} }  n^{- {2r \over 2r+1} }.
\end{equation}

Consider the second sum in (\ref{equ:sum}). We show two technical results. Firstly,
since the design points corresponding to $\Delta\bar y_{d_n+4, j} - \Delta \bar y_{d_n-2, j}$ are disjoint for different $j$,
$\Delta\bar y_{d_n+4, j} - \Delta \bar y_{d_n-2, j}$ are independent. Hence,
for $j>J+K$,
\begin{align*}
 \mathbb E (I_j)  &\le \prod_{i=J+K}^{j-1} P\Big(\Delta\bar y_{d_n+4, i}  -\Delta \bar y_{d_n-2, i} > \lambda 2^{ {i \over 2}+1} n^{-{2 \over 5} }\sigma \Big) \\
&\le \Big\{ P\Big(Z>\lambda - {1\over 2^{K/2 }}\Big) \Big\}^{j-J-K} \, \le \, \gamma^{j-J-K}.
\end{align*}
Secondly, in view of the argument for (\ref{eqn:wt_f-grave f}) and Proposition \ref{lem:pointwise}, we have $\mu \in[0, 1]$ and a constant $c_{12}>0$ such that
\begin{align*}
  \lefteqn{ \mathbb E |\wt f_j(x_0) - \grave f(x_0) |^4 } \\
  & \qquad  \, \le \, 8 \big(  \mu^4 \mathbb E | \hat f(\zeta_{d_n}) - \bar f(\zeta_{d_n})|^4 + (1-\mu)^4  \mathbb E | \hat f(\zeta_{d_n+1}) - \bar f(\zeta_{d_n+1})|^4 \big)  \nonumber \\
  & \qquad \, \le \, 8 C_{4} \sigma^4 K^2_{n,j} \cdot n^{-2} \, \le \, c^2_{12} \sigma^4 \cdot \big( 2^j  \cdot n^{-{4 \over 5} } \big)^2. 
\end{align*}
Using these results and $\gamma\in (0,1/4)$, we obtain a constant $c_{13}>0$ such that
\begin{align*}
& \quad \sum_{j=J+K+1}^\infty \mathbb E \Big( (\wt f(x_0) - \grave f(x_0))^2 I_j \Big) \\
& \le \sum_{j=J+K+1}^\infty \Big(\mathbb E |\wt f(x_0) - \grave f(x_0)|^4 \cdot \mathbb E (I_j^2) \Big)^{1/2}\, \le \, \sum_{j=J+K+1}^\infty c_{12} 2^j n^{-{4 \over 5} } \sigma^2 \cdot \gamma^{{ (j-J-K) \over 2 } } \\
&\le \sum_{j=J+K+1}^\infty c_{12} (4\gamma)^{(j-J)/2} \gamma^{-K/2} 2^J n^{-{4 \over 5} }\sigma^2 \, \le \, c_{13} 2^J n^{-{4 \over 5} }\sigma^2.
\end{align*}
%
Furthermore, since $f$ is convex,  $\tau_j= \mathbb E(\Delta\bar y_{d_n+4, j} - \Delta \bar y_{d_n-2, j})$ is a  decreasing function  of $j$. Therefore,  in view of $0\le \grave f(x_0) - f(x_0) \le 2 \tau_j$ and (\ref{equ:tauJ}),
$$ \sum_{j=J+K+1}^\infty \mathbb E \Big( (\grave f(x_0) - f(x_0))^2 I_j  \Big) \, \le \, \sum_{j=J+K+1}^\infty  4 \tau_J^2 \cdot \mathbb E (I_j)  \, \le \, 4 \tau_J^2 \le 16 \cdot 2^J n^{-{4 \over 5} }\sigma^2.$$
By virtue of the above results and (\ref{equ:J2}), we obtain  $c_{14}>0$ such that
\begin{equation}\label{equ:sum2}
\sum_{j=J+K+1}^\infty \mathbb E [(\wt f(x_0) - f(x_0))^2 I_j]  \, \le \, c_{14} L^{{ 2 \over 2r+1} }\sigma^{{ 4r \over 2r+1} }  n^{-{ 2r \over 2r+1} }.
\end{equation}
Hence, the theorem follows by combining (\ref{equ:sum1}) and (\ref{equ:sum2}).
\end{proof}

%
\section{Proof of Theorem \ref{thm:sig}} \label{sect:proof_sig}

Recall that $\hat f_y := \big( \hat f^{[p]} (x_1), \ldots, \hat f^{[p]}(x_n)\big)^T$ with coefficient matrices $A_\alpha $ defined in (\ref{eqn:A_alpha})  and
$\vec f :=  \big( f(x_1), \ldots, f(x_n)\big)^T$.  

\begin{lemma}\label{lem:sig1}
%
%
Fix a spline degree $p$.
The following hold:
\begin{enumerate}
  \item [(1)]  For any index set $\alpha$, $0 \le \mathrm{trace}(A_\alpha)  \le c_{\infty, p} (K_n+p)$;
  \item [(2)]
$
\mathbb E\big[\big\langle y - \vec f , \hat f_y - \vec f~ \big\rangle\big] \, = \, \sigma^2 \mathbb E \big[\mathrm{trace}\big(A_{\alpha(y)}  \big)\big] \, \le \,
  c_{\infty, p} \, (K_n+p) \, \sigma^2.
$
  \end{enumerate}
\end{lemma}

\begin{proof}
(1) Since $A_\alpha$ is symmetric positive semidefinite, its trace is nonnegative.
 For a given $\alpha$, recall that
$G_\alpha := F^T_\alpha
(F_\alpha \Lambda F^T_\alpha)^{-1} F_\alpha$. Hence $A_\alpha = X
G_\alpha X^T/\beta_n$.
%
%
By virtue of $\| G_\alpha \|_\infty \le c_{\infty,
p}$ from Theorem~\ref{theorem:uniform_Lip} and  $|(\Lambda)_{ij}| \le 1$ for all $i, j$ (by the
definition of $\beta_n$), we have, for any $\alpha$,
$
  0 \le \mathrm{trace}\big( A_\alpha \big) = \mathrm{trace} \Big( G_\alpha \cdot \frac{X^T X}{\beta_n} \Big)
  = \mathrm{trace} ( G_\alpha \Lambda ) \le c_{\infty, p} (K_n + p).
$

(2)  Since  $\hat f_y:\mathbb R^n \rightarrow
\mathbb R^n$ is continuous and piecewise linear, it admits a conic subdivision of
$\mathbb R^n$ \cite{FPang03_book, Scholtes94_thesis}, i.e., there
exist a finite collection of polyhedral cones  $\{ \mathcal C_j
\}^\ell_{j=1}$ and linear functions $\{ g^j \}^\ell_{j=1}$  satisfying the similar conditions as specified in the proof of Proposition~\ref{lem:pointwise}. In particular, each cone $\mathcal C_j$ has nonempty interior and
%
%
$\hat f_y$ coincides with $ g^j$ on each $\mathcal C_j$. Clearly,
$g^j(y)= A_\alpha y$ for some index set $\alpha$. In this case, we write the cone
$\mathcal C_j$ as $\mathcal C_\alpha$. Let $\mbox{int} (\mathcal C_\alpha)$
denote the interior of $ \mathcal  C_\alpha$.
Obviously, $\hat f_y$ is differentiable on $\mbox{int}( \mathcal
C_\alpha)$. Indeed, the (Fr\'{e}chet-)derivative of $\hat f_y$ is
$A_\alpha$ for any $y \in \mbox{int}( \mathcal C_\alpha)$.
Let $h(y):= \hat f_y-\vec f$. Since $\mathbb R^n \setminus \big( \bigcup_j \mbox{int}
(\mathcal C_j) \big)$ has zero measure,  $h$ is almost
differentiable on $\mathbb R^n$ in the sense of \cite[Definition
1]{stein_81}. Let $\phi(\bf z)$ be the standard normal density on
$\mathbb R^n$ with variance $\sigma^2$. We have
\begin{align*}
   \mathbb E\Big(\Big \| \frac{\partial h}{\partial y}(y) \Big\| \Big) &= \int_{ \bigcup_j
\mbox{int} (\mathcal C_j) } \Big \| \frac{\partial h}{\partial y}
\big({\bf z}+\vec f \big) \Big\| \phi( {\bf z}) d {\bf z} \\
& = \sum_\alpha
\int_{{\bf z}+\vec f \, \in  \, \mbox{int}(\mathcal C_\alpha)} \|A_\alpha\|
\phi({\bf z}) d {\bf z} \le
 \max_\alpha \| A_\alpha\| <\infty.
\end{align*}
Letting $Z:=\sigma(\epsilon_1, \ldots, \epsilon_n)^T$, we have $\mathbb E[\langle y - \vec f, \hat f_y - \vec f\rangle] = \mathbb E[\langle
Z, h(y) \rangle] = \mathrm{trace}\big( \mathbb E[Z \cdot h^T(y) ]\big)$.
By the above results and Stein's Lemma \cite[Lemma 2]{stein_81}, 
\begin{align*}
 \lefteqn{ \mathrm{trace}\big( \mathbb E[Z \cdot h^T(y) ]\big) }  \\
  & = \sigma^2
  \mathrm{trace}\Big(\mathbb  E \Big[ \frac{\partial h}{\partial y}(y) \Big]  \Big)
   =  \sigma^2 \int_{ \bigcup_j \mbox{int} (\mathcal C_j) } \mathrm{trace}\Big( \frac{\partial h}{\partial
  y} ({\bf z}+\vec f ) \Big) \phi( {\bf z})   d {\bf z} \\
  & =  \sigma^2 \sum_\alpha \int_{} \mathrm{trace}\Big( \frac{\partial h}{\partial
  y} ({\bf z}+\vec f ) \Big) \phi( {\bf z}) \cdot I_{ \{ {\bf z} | {\bf z} + \vec f \in \mbox{int}(\mathcal C_\alpha) \} } d {\bf z} \\
  & =  \sigma^2 \sum_\alpha \int_{} \mathrm{trace}\big( A_\alpha \big) \phi( {\bf z}) \cdot I_{ \{ {\bf z} | {\bf z} + \vec f \in \mbox{int}(\mathcal C_\alpha) \} } d {\bf
  z}  \, =\, \sigma^2 \mathbb E\big (\mathrm{trace}(A_{\alpha(y)} ) \big).   
%
%
\end{align*}
%
%
Statement (2) thus follows from (1).
\end{proof}

Equipped with the above lemma,  we have the proof of Theorem \ref{thm:sig} below.

\begin{proof}[Proof of Theorem \ref{thm:sig}]
The MLE of $\sigma^2$ is $\hat\sigma^2 = \|y - \hat f_y\|_2^2/n$.
%
%
Let $R_n :=  \|y - \vec f\|_2^2-\|y-\hat f_y\|_2^2 $ such that $ \hat \sigma^2 =  \|y - \vec f\|_2^2/n - R_n/n$.
%
%
Since
\begin{eqnarray*}
 R_n \, = \,   \|y - \vec f\|_2^2-\|y-\hat f_y\|_2^2
\, =\,   2\langle y - \vec f, \hat f_y - \vec f \, \rangle - \|\hat f_y - \vec f\|_2^2,
\end{eqnarray*}
we have, by Lemma \ref{lem:sig1} and (\ref{eqn:statement_2}),
\begin{align*}
 \big| \mathbb E(R_n) \big | &  \le \,  2\sigma^2 \mathbb E(\mathrm{trace}(A_{\alpha(y)})) + \mathbb E \|\hat f_y - \vec f\|_2^2  \\
 & \le 2\sigma^2 c_{\infty,p} (K_n+p) + n \big( C^2_{1r} L^2 K^{-2r}_n + C_{3r} \sigma^2 K_n \, n^{-1} \big),
\end{align*}
where $p=\lceil r-1 \rceil$. This shows that $|\mathbb E(R_n)| = O(K_n + n K^{-2r}_n)$.
Hence, we deduce that (i) if $K_n=o(n)$ with $K_n\rightarrow\infty$ as $n\rightarrow \infty$, then  $\hat \sigma^2 \rightarrow \sigma^2$ in probability; (ii) if $K_n=o(\sqrt{n})$ with $K_n\rightarrow\infty$ as $n\rightarrow \infty$, then $\sqrt{n}( \hat\sigma^2 - \sigma^2)$ is asymptotically normal with mean zero and variance $2\sigma^4$; and (iii) if $K_n$ is of order $n^{1 \over 2r+1}$, then $ \big | \mathbb E( \hat\sigma^2 - \sigma^2) \big |$ is of order $n^{-2 r \over 2r+1}  $.
\end{proof}




\begin{thebibliography}{99}


{\small

\bibitem{barlow_72} {\sc Barlow, R.E., Bartholomew, D. J. and Brunk, H. D.}  (1972).  {\it Statistical Inference Under
Order Restrictions: The Theory and Application of Isotonic
Regression}. Wiley, New York.

\bibitem{bertin_04} {\sc Bertin, K.} (2004). Minimax exact constant in sup-norm for nonparametric regression with random design. {\it Journal of Statistical Planning and Inference}, {\bf 123}, 225-242.


\bibitem{birke_06}
{\sc Birke, M. and Dette, H.} (2006). Estimating a convex function
in nonparametric regression. {\it Scandinavian Journal of
Statistics}, {\bf 34}, 384-404.

\bibitem{brown_96} {\sc Brown, L. and Low, M.} (1996). A constrained risk inequality with applications to nonparametric estimation. {\it Annals of Statistics}, {\bf 24}, 2524-2535.

\bibitem{brunk_58} {\sc Brunk, H.D.} (1958). On the estimation of parameters
restricted by inequalities. {\it Annals of Mathematical
Statistics}, {\bf 29}, 437-454.

\bibitem{cator_11} {\sc Cator, E.} (2011). Adaptivity and optimality of the monotone least squares estimator. {\it Bernoulli}, {\bf 17}, 714-735.

\bibitem{CPStone_book92}
{\sc Cottle, R. W., Pang, J. S. and Stone, R.E.} (1992). {\it The
Linear Complementarity Problem}. Academic Press Inc., Cambridge.

\bibitem{Demko_siam77}
{\sc Demko, S.} (1977). Inverses of band matrices and local convergence of
spline projections. {\it SIAM Journal on Numerical Analysis},
{\bf 14}, 616--619.

%
%

\bibitem{donoho_94} {\sc Donoho, D. L.} (1994). Asymptotic minimax risk for sup-norm loss: solution via optimal recovery. {\it Probability Theory Related Fields}, {\bf 99}, 145-170.

%
%

\bibitem{dumbgen_03} {\sc D\"umbgen, L.} (2003). Optimal confidence bands for shape-restricted curves. {\it Bernoulli}, {\bf 9}, 423-449.

\bibitem{dumbgen_04} {\sc D\"umbgen, L., Freitag, S. and Jongbloed, G.} (2004). Consistency of concave regression with an application to current-status data. {\it Mathematical Methods of Statistics}, {\bf 13}, 69-81.


\bibitem{dumbgen_01}   {\sc D\"umbgen, L. and Spokoiny, V. G.} (2001). Multiscale testing of qualitative hypotheses. {\it Annals of Statistics}, {\bf 29}, 124-152.


\bibitem{FPang03_book}
{\sc Facchinei, F. and Pang, J.S. } (2003). {\it
Finite-Dimensional Variational Inequalities and Complementarity
Problems}. Springer-Verlag, New York.

\bibitem{groeneboom_01}
{\sc Groeneboom, P., Jongbloed, G. and Wellner, J. A.} (2001)
Estimation of a convex function: Characterizations and asymptotic
theory. {\it Ann. Statist.}, {\bf 29}, 1653-1698.

\bibitem{hall_01} {\sc Hall, P. and Huang, L.} (2001). Nonparametric kernel regression subject to monotonicity constraints. {\it Annals of Statistics}, {\bf 29}, 625-647.

\bibitem{hanson_76}
{\sc Hanson, P.L. and Pledger, G.} (1976). Consistency in concave
regression. {\it Annals of Statistics}, {\bf 4}, 1038-1050.

%
%

\bibitem{hildreth_54}
Hildreth, C. (1954). Point estimates of ordinates of concave
functions. {\it Journal of American Statistical Association}, {\bf
49}, 598-619.


\bibitem{HornJohson_book85}
{\sc Horn, R.A.  and Johnson, C.R.} (1985). {\it Matrix Analysis}. Cambridge
University Press.

\bibitem{kiefer_82} {\sc Kiefer, J.} (1982). Optimum rates for non-parametric density and regression estimates, under order restrictions. In: Kallianpur, G., Krishnaiah, P.R., Ghosh, J.K. (Eds.), Statistics and Probability: Essays in honor of C.R. Rao. North-Holland, Amsterdam, 419-428.

\bibitem{korostelev_93} {\sc Korostelev, A.P.} (1993). Exact asymptotically minimax estimator for nonparametric regression in uniform norm. {\it Theory Probab. Appl.}, {\bf 38}, 737-743.

\bibitem{lepski_90} {\sc Lemskii, O.V.} (1990). On problem of adaptive estimation in Gaussian white noise. {\it Theory Probab. Appl.}, {\bf 35}, 454-466.

%
%

\bibitem{lepski_92} {\sc Lemskii, O.V.} (1992). On problems of adaptive estimation in white Gaussian noise. In R.Z. Khasminskij (ed.), {\it Topics in Nonparametric Estimation}, Adv. Soviet Math.  {\bf 12}, 87-106. Providence, RI: American Mathematical society.

\bibitem{lepski_97} {\sc Lepski, O. V. and Spokoiny, V.G.} (1997). Optimal pointwise adaptive methods in nonparametric estimation. {\it Annals of Statistics}, {\bf 25}, 2512-2546.

\bibitem{lepski_00} {\sc Lepski, O. V. and Tsybakov, A. B.} (2000). Asymptotically exact nonparmetric hypothesis testing in sup-norm and at a fixed point. {\it Probability Theory Related Fields}, {\bf 117}, 17-48.

\bibitem{low_02} {\sc Low, M. and Kang, Y.} (2002). Estimating monotone functions. {\it Statistics \& Probability Letters}, {\bf 56}, 361-367.

\bibitem{mammen_91}
{\sc Mammen, E.} (1991). Nonparametric regression under
qualitative smoothness assumptions. {\it Ann. Statist.}, {\bf 19},
741-759.

\bibitem{mammen_99}
{\sc Mammen, E. and Thomas-Agnan, C.} (1999). Smoothing splines
and shape restrictions. {\it Scandinavian Journal of Statistics},
{\bf 26}, 239-252.

\bibitem{meyer_00}
{\sc Meyer, M. and Woodroofe, M.} (2000). On the degrees of freedom of shape-restricted regression. {\it Annals of Statistics}, {\bf 28},
1083-1104.

\bibitem{meyer_08}
{\sc Meyer, M.} (2008). Inference using shape-restricted
regression splines. {\it Annals of Applied Statistics}, {\bf 2},
1013-1033.



\bibitem{mukerjee_88} {\sc Mukerjee, H.} (1988). Monotone nonparametric
regression. {\it Annals of Statistics}, {\bf 16}, 741-750.

\bibitem{Nemirovski_note} {\sc Nemirovski, A.} (2000). Topics in Non-Parametric Statistics. {\it Lecture on Probability Theory and Statistics}. Berlin, Germany: Springer-Verlag, Vol 1738, Lecture Notes in Mathematics.

\bibitem{ramsay_88} {\sc Ramsay, J.O.}  (1988). Estimating smooth monotone functions. {\it J. Roy. Statist. c. Ser. B}, {\bf 60}, 365-375.

\bibitem{rice_84} {\sc Rice, J.} (1984). Bandwidth choice for nonparametric regression. {\it Annals of Statistics}, {\bf 12}, 1215-1230.

\bibitem{robertson_88}
{\sc Robertson, T., Wright, F. T. and Dykstra, R. L.} (1988). {\it
Order Restricted Statistical Inference}. Wiley Series in
Probability and Mathematical Statistics.

\bibitem{Scholtes94_thesis}
{\sc Scholtes, S.} (1994). {\it Introduction to Piecewise Differentiable
Equations}. Habilitation thesis, Institut f\"ur Statistik und
Mathematische Wirtschaftstheorie, Universit\"at Karlsruhe.

\bibitem{ShenPang_SICON05}
{\sc Shen, J. and Pang, J.S.} (2005). Linear complementarity systems: Zeno states. {\it SIAM Journal on Control and Optimization}, {\bf 44(3)}, 1040--1066.


\bibitem{ShenWang_SICON11}
{\sc Shen, J. and Wang, X.} (2011). Estimation of monotone functions via $P$-spline: A constrained dynamical optimization approach. {\it SIAM Journal on Control and Optimization}, {\bf 49(2)}, 646--671.

\bibitem{ShenWang_CDC11}
{\sc Shen, J. and Wang, X.} (2011). A constrained optimal control approach to smoothing splines. {\it Proc. of the 50th IEEE Conf. on Decision and Control}, 1792--1734, Orlando, Florida.


\bibitem{stein_81}
{\sc Stein, C.} (1981). Estimation of the mean of a multivariate normal distribution. {\it Annals of Statistics}, {\bf 9}, 1135-1151.



\bibitem{stone_82}
{\sc Stone, C.J. }(1982). Optimal rate of convergence for
nonparametric regression. {\it Annals of Statistics}, {\bf 10},
1040-1053.

\bibitem{tsybabov_98} {\sc Tsybakov, A.B.} (1998). Pointwise and sup-norm sharp adaptive estimation of functions on the Sobolev classes. {\it Annals of Statistics}, {\bf 26},
2420-2469.

\bibitem{tsybakov_10} {\sc Tsybakov, A.B.} (2010). {\it Introduction to Nonparametric Estimation}. Springer.

\bibitem{wang_10} {\sc Wang, X. and Shen, J.} (2010). A class of grouped Brunk estimators and penalized spline estimators for monotone regression.  {\it Biometrika}, {\bf 97}, 585-601.

%
%

\bibitem{woodroofe_93}
{\sc Woodroofe, M. B. and Sun, J.} (1993). A penalized maximum
likelihood estimate of $f(0_+)$ when $f$ is nonincreasing. {\it
Statistica Sinica}, {\bf 3}, 501-515.



\bibitem{wright_81}
{\sc Wright, F.T. }(1981). The asymptotic behavior of monotone
regression estimates. {\it Annals of Statistics}, {\bf 9},
443-448.

%
%

\bibitem{zhou_98}
{\sc Zhou, S., Shen, X. and Wolfe, D.A.} (1998). Local asymptotics for
regression splines and confidence regions. {\it Annals of
Statistics}, {\bf 26}, 1760-1782.

\bibitem{zhou_00} {\sc Zhou, S. and Wolfe, D.} (2000). On derivative estimation in spline regression. {\it Statistica Sinica}, {\bf 10}, 93-108.
}

\end{thebibliography}
\end{document}